\documentclass[11pt]{amsart}
\usepackage[margin=1.45in]{geometry}

\pdfoutput=1

\usepackage[english]{babel}
\usepackage[utf8]{inputenc}
\usepackage[T1]{fontenc}
\usepackage{microtype}
\usepackage{amsmath}
\usepackage{amsfonts}
\usepackage{amssymb}
\usepackage{amsthm}
\usepackage{eucal}
\usepackage{comment}
\usepackage{hyperref}
\usepackage{tikz-cd}
\tikzcdset{column sep/normal=1.75em}
\tikzcdset{row sep/normal=1.75em}
\usetikzlibrary{babel}
\usepackage{relsize}
\usepackage{enumitem}

\newtheorem{proposition}{Proposition}[section]
\newtheorem{lemma}[proposition]{Lemma}
\newtheorem{theorem}[proposition]{Theorem}
\newtheorem{corollary}[proposition]{Corollary}

\DeclareMathOperator{\Projectives}{\mathcal{P}} 
\DeclareMathOperator{\Injectives}{\mathcal{I}} 
\DeclareMathOperator{\proj}{\mathsf{proj}} 
\DeclareMathOperator{\Ker}{\mathsf{Ker}} 
\DeclareMathOperator{\Image}{\mathsf{Im}} 
\DeclareMathOperator{\Hom}{\mathsf{Hom}} 
\DeclareMathOperator{\End}{\mathsf{End}} 
\DeclareMathOperator{\add}{\mathsf{add}} 
\DeclareMathOperator{\gldim}{\mathsf{gl.dim}} 
\DeclareMathOperator{\domdim}{\mathsf{dom.dim}} 
\DeclareMathOperator{\Ext}{\mathsf{Ext}} 
\DeclareMathOperator{\op}{\mathsf{op}} 
\DeclareMathOperator{\Ab}{\mathsf{Ab}} 
\DeclareMathOperator{\pdim}{\mathsf{pd}} 
\DeclareMathOperator{\mr}{\mathsf{m}} 
\DeclareMathOperator{\SplitEpi}{\mathsf{Split \text{ } Epi}} 
\DeclareMathOperator{\SplitMono}{\mathsf{Split \text{ } Mono}} 
\DeclareMathOperator{\Epi}{\mathsf{Epi}} 
\DeclareMathOperator{\Mono}{\mathsf{Mono}} 
\DeclareMathOperator{\E}{\mathsf{E}} 
\let\mod\relax
\DeclareMathOperator{\mod}{\mathsf{mod}} 
\DeclareMathOperator{\Mod}{\mathsf{Mod}} 
\DeclareMathOperator{\Mor}{\mathsf{Mor}} 
\newcommand{\A}{\mathcal{A}}
\newcommand{\B}{\mathcal{B}}
\newcommand{\C}{\mathcal{C}}

\DeclareMathOperator{\eff}{\mathsf{eff}} 
\DeclareMathOperator{\modone}{\mathsf{mod}_{1} \hspace{-0.1em}} 

\begin{document}

\title{A functorial approach to $0$-abelian categories}
\author{Vitor Gulisz}
\address{Mathematics Department, Northeastern University, Boston, MA 02115, USA}
\email{gulisz.v@northeastern.edu}
\date{September 29, 2025}

\begin{abstract}
We use functorial methods to define and study $0$-abelian categories, which we propose to be the case $n = 0$ of Jasso's $n$-abelian categories. In particular, we define a bifunctor for $0$-abelian categories with enough injectives or projectives, which is analogous to the extension bifunctor for an abelian category. We prove a few results concerning this bifunctor, including $0$-abelian versions of the long exact sequence involving the extension bifunctors, of a conjecture due to Auslander on the direct summands of the extension functors, and of the Hilton-Rees theorem. These results are then applied to the study of the stable categories of a $0$-abelian category, and a similar discussion is carried out for the stable categories of an abelian category. Moreover, by specializing our results to modules over rings, we show that $0$-abelian categories with additive generators are in correspondence with semi-hereditary rings. We present applications to these rings.
\end{abstract}

\maketitle

\tableofcontents

\section{Introduction}\label{section.1}

This paper is a continuation of \cite{2409.10438}, and its \textit{raison d'être} is to extend the concept of an $n$-abelian category to the case $n = 0$. Let us start with an overview.

Recall that the notion of an $n$-abelian category was introduced by Jasso in \cite{MR3519980} for a positive integer $n$, which is based on the concepts of $n$-kernel and $n$-cokernel. These definitions were given in such a way that, when $n = 1$, they coincide with the classical notions of an abelian category, kernel and cokernel, respectively. In this context, it is natural to ask if it would also be possible to consider the case $n = 0$, which should lead to the notions of a ``$0$-abelian category'', ``$0$-kernel'' and ``$0$-cokernel''. As a first attempt to obtain these concepts, perhaps, one could try to set $n = 0$ in the relevant definitions and in the axioms of an $n$-abelian category given in \cite{MR3519980}. However, this approach does not provide any clear answer to what should happen when $n = 0$, and only leaves room for guesses. So, one might want to look for other alternatives to tackle this problem. As one such alternative, in this paper, we suggest to extend the functorial approach to $n$-abelian categories developed in \cite{2409.10438} to the case $n = 0$, and use it as a guide to develop a theory of ``$0$-abelian categories''. Luckily, this strategy turns out to be quite straightforward. Let us explain why.

Among the characterizations of $n$-abelian categories given in \cite{2409.10438}, the one presented below seems to be the easiest to state, which follows from \cite[Theorems 3.6 and 6.10]{2409.10438}.

\begin{theorem}\label{theorem.1}
An additive and idempotent complete category $\C$ is $n$-abelian if and only if $\C$ satisfies the following axioms: \begin{enumerate}[leftmargin=3.8em]
    \item[\textup{($n$F1)}] $\C$ is right coherent and $\gldim (\mod \C) \leqslant n+1$.
    \item[\textup{($n$F1$^{\op}$)}] $\C$ is left coherent and $\gldim (\mod \C^{\op}) \leqslant n+1$.
    \item[\textup{($n$F2)}] Every $F \in \mod \C$ with $\pdim F \leqslant n$ is a syzygy.
    \item[\textup{($n$F2$^{\op}$)}] Every $F \in \mod \C^{\op}$ with $\pdim F \leqslant n$ is a syzygy.
\end{enumerate}
\end{theorem}

Now, although $n$ stands for a positive integer in Theorem \ref{theorem.1}, it makes sense to consider its statement for $n = 0$, and we can use it as a definition of a \mbox{``$0$-abelian category''.} So, let $\C$ be an additive and idempotent complete category. We say that $\C$ is a \textit{$0$-abelian category} if, in addition to being additive and idempotent complete, $\C$ satisfies the axioms ($0$F1), ($0$F1$^{\op}$), ($0$F2) and ($0$F2$^{\op}$). However, since ($0$F2) and ($0$F2$^{\op}$) are empty statements, as every projective object is a syzygy of itself, we conclude that ($0$F1) and ($0$F1$^{\op}$) are the only axioms to be considered when $n = 0$. Also, recall from \cite[Proposition 3.2]{2409.10438} that, when $n$ is positive, the axioms ($n$F1) and ($n$F1$^{\op}$) for $\C$ are equivalent to the conditions that $\C$ has $n$-kernels and $n$-cokernels, respectively. Therefore, we should expect that a similar result holds for $n = 0$, and this observation motivated our definition of ``$0$-kernel'' and ``$0$-cokernel'', which we give in Section \ref{section.2}. With these definitions, we prove in Proposition \ref{proposition.1} that the axioms ($0$F1) and ($0$F1$^{\op}$) for $\C$ are equivalent to $\C$ having ``$0$-kernels'' and ``$0$-cokernels'', respectively. Finally, we conclude in Theorem \ref{theorem.3} that a category is $0$-abelian if and only if it is additive and has ``$0$-kernels'' and ``$0$-cokernels''.


While the above discussion is carefully carried out in Subsection \ref{subsection.2.1}, we now outline the subsequent content of this paper.

Once the notion of a $0$-abelian category is settled, we continue Section \ref{section.2} with some basic results on it. For instance, we prove in Corollary \ref{corollary.1} that if a subcategory of a $0$-abelian category is closed under finite direct sums and direct summands, then it is also $0$-abelian. Moreover, we consider ``$0$-exact sequences'' in $0$-abelian categories, and we extend a few results on $m$-exact sequences in $n$-abelian categories and on the von Neumann regularity of $n$-abelian categories to the case $n = 0$.

In Section \ref{section.3}, using the functorial approach, we study how the factorization of a morphism in an abelian category as an epimorphism followed by a monomorphism compares to similar factorizations of a morphism in a $0$-abelian category. We then use this comparison to suggest the existence of analogous factorizations of a morphism in an $n$-abelian category when $n \geqslant 2$. The discussion carried in this section hints that the case $n = 0$ of $n$-abelian categories might help to understand the case $n \geqslant 2$.

In Section \ref{section.4}, we consider $0$-abelian categories with enough injectives or projectives. First, we extend results about $n$-abelian categories from \cite[Section 7]{2409.10438} to $n = 0$. Then we define a bifunctor $\E_{\C}(-,-)$ for a $0$-abelian category $\C$ with enough injectives or projectives, and we give a few results that support the idea that $\E_{\C}(-,-)$ is analogous to the extension bifunctor $\Ext_{\A}^{1}(-,-)$ for an abelian category $\A$. Among these results, we prove in Theorem \ref{theorem.4} that every ``$0$-exact sequence'' in $\C$ induces a ``not so long exact sequence'' involving the bifunctor $\E_{\C}(-,-)$, which is analogous to the long exact sequence involving the bifunctors $\Ext_{\A}^{i}(-,-)$ induced by a short exact sequence in $\A$. Moreover, in Theorem \ref{theorem.5}, we prove a $0$-abelian version of a conjecture attributed to Auslander, on the direct summands of the extension functors on an abelian category. Also, in Theorem \ref{theorem.6}, we present a $0$-abelian version of the Hilton-Rees theorem, and we use this and previous results to study the stable categories of a $0$-abelian category. We then close the section with analogous results on the stable categories of an abelian category and with sources of examples to which these results can be applied.

In Section \ref{section.5}, we restrict results obtained thus far to rings and modules over rings. By doing so, we show in Theorem \ref{theorem.9} that $0$-abelian categories with additive generators are in correspondence with semi-hereditary rings, and we give a couple of examples. Furthermore, we use the theory of $0$-abelian categories to prove a few known results on semi-hereditary rings.

In Appendix \ref{section.6}, we prove a couple of results on ``$0$-kernels'' and ``$0$-cokernels''. In particular, in Proposition \ref{proposition.10}, we describe these concepts in terms of universal properties, which are similar to the universal properties of kernels and cokernels.

Finally, in Appendix \ref{section.7}, we show how $0$-abelian categories relate to $0$-Auslander categories, and we question whether there exists a notion of ``$0$-cluster tilting subcategory''.

\subsection{Conventions and preliminaries}\label{subsection.1.1}

We will follow the same conventions, definitions, notations and preliminaries of \cite{2409.10438}. In particular, all subcategories are assumed to be full and, throughout this paper, $\C$ will be an additive and idempotent complete category. However, from now on, $n$ will be a nonnegative integer.

In a few instances in this text, we will suggest the use of a result that says that any pair of composable morphisms in an abelian category induces an exact sequence involving the kernels and cokernels of the given morphisms and their composition. Its precise statement and proof can be found in \cite[Theorem 1]{MR202792} or \cite[Satz 9]{MR170924}. We will refer to this result as the \textit{circular sequence}, and we hope to help popularize it.

\section{\texorpdfstring{$0$}{0}-Abelian categories}\label{section.2}

Motivated by the functorial approach to $n$-abelian categories developed in \cite{2409.10438} for a positive integer $n$, in this section, we introduce the notions of a ``$0$-abelian category'', ``$0$-kernel'' and ``$0$-cokernel''. Once these are established, we explain how $0$-abelianness behaves with subcategories, and we extend a few known results on the von Neumann regularity and $m$-exact sequences of $n$-abelian categories for \mbox{$n$ positive to the case $n = 0$.}


\subsection{Their axioms}\label{subsection.2.1}

Following our previous discussion in Section \ref{section.1}, we say that an additive and idempotent complete category $\C$ is \textit{$0$-abelian} if it satisfies the following axioms:
\begin{enumerate}[leftmargin=3.8em]
    \item[\textup{($0$F1)}] $\C$ is right coherent and $\gldim (\mod \C) \leqslant 1$.
    \item[\textup{($0$F1$^{\op}$)}] $\C$ is left coherent and $\gldim (\mod \C^{\op}) \leqslant 1$.
\end{enumerate}
It is convenient to remark that if $\C$ is coherent, then the global dimensions of $\mod \C$ and $\mod \C^{\op}$ coincide, see \cite[Theorem A.5]{2409.10438}. Thus, we can simply say that $\C$ is $0$-abelian if $\C$ is coherent and the global dimension of $\mod \C$ is at most $1$.

Let us briefly recall that, in \cite[Propositions 3.2 and 3.5]{2409.10438}, we started with the axioms of an $n$-abelian category given by Jasso in \cite{MR3519980} (for $n$ positive), and then we ``translated'' them in terms of $\mod \C$ and $\mod \C^{\op}$. Now, for the case $n = 0$, we are starting with the axioms of a $0$-abelian category in terms of $\mod \C$ and $\mod \C^{\op}$, and it would be desirable to ``translate'' them back to axioms which are similar to those given by Jasso. In order to do this, we first need to understand what a ``$0$-kernel'' and a ``$0$-cokernel'' are, and for this, we will use the functorial approach.

To begin with, let us recall how $n$-kernels and $n$-cokernels in $\C$ can be understood through $\mod \C$ and $\mod \C^{\op}$, respectively, when $n$ is a positive integer. Let $f \in \C(X,Y)$ be a morphism in $\C$. As it was discussed in \cite[Section 3]{2409.10438}, if $n$ is a positive integer and $\C$ has $n$-kernels, that is, if $\C$ is right coherent and $\gldim (\mod \C) \leqslant n + 1$, see \cite[Proposition 3.2]{2409.10438}, then an $n$-kernel of $f$ is obtained as follows: First, consider the exact sequence \[ \begin{tikzcd}
{\C(-,X)} \arrow[r, "{\C(-,f)}"] &[1.2em] {\C(-,Y)} \arrow[r] & F \arrow[r] & 0
\end{tikzcd} \] in $\mod \C$, which is obtained by taking the cokernel of $\C(-,f)$. Since the above sequence is a projective presentation of $F$ in $\mod \C$ and $\pdim F \leqslant n + 1$, we can complete it to a projective resolution \[ \begin{tikzcd}
0 \arrow[r] &[-0.15em] {\C(-,X_{n})} \arrow[r] &[-0.15em] \cdots \arrow[r] &[-0.15em] {\C(-,X_{1})} \arrow[r] &[-0.15em] {\C(-,X)} \arrow[r, "{\C(-,f)}"] &[1.2em] {\C(-,Y)} \arrow[r] &[-0.15em] F \arrow[r] &[-0.15em] 0
\end{tikzcd} \] of $F$ in $\mod \C$. By the Yoneda lemma, this resolution leads to a sequence of morphisms \[ \begin{tikzcd}
X_{n} \arrow[r] & \cdots \arrow[r] & X_{1} \arrow[r] & X
\end{tikzcd} \] in $\C$, which is, by definition, an $n$-kernel of $f$. Conversely, every $n$-kernel of $f$ induces a projective resolution of $F$ in $\mod \C$ as above, which extends the projective presentation of $F$ in $\mod \C$ that we started with. If $\C$ has $n$-cokernels, then the $n$-cokernels of $f$ can be characterized in a similar way, by considering $\mod \C^{\op}$ instead of $\mod \C$.

Now, for $n = 0$, suppose that $\C$ is right coherent and $\gldim (\mod \C) \leqslant 1$. Given a morphism $f \in \C(X,Y)$ in $\C$, we can again consider the exact sequence \[ \begin{tikzcd}
{\C(-,X)} \arrow[r, "{\C(-,f)}"] &[1.2em] {\C(-,Y)} \arrow[r] & F \arrow[r] & 0
\end{tikzcd} \] in $\mod \C$, which is the beginning of a projective resolution of $F$ in $\mod \C$. However, in this case, we have $\pdim F \leqslant 1$. So, instead of extending the above sequence to a projective resolution of $F$, we will reduce it to a projective resolution \[ \begin{tikzcd}
0 \arrow[r] & {\C(-,Z)} \arrow[r, "{\C(-,g)}"] &[1.2em] {\C(-,Y)} \arrow[r] & F \arrow[r] & 0
\end{tikzcd} \] of $F$ in $\mod \C$ by taking the kernel of $\C(-,Y) \to F$. This leads to a morphism $g \in \C(Z,Y)$ in $\C$, which should give the notion of a ``$0$-kernel'' of $f$.

Let us characterize the morphism $g$ obtained above. Under the conditions of the previous paragraph, observe that $\C(-,g)$ is nothing but the image of $\C(-,f)$ in $\mod \C$, hence there is a factorization $\C(-,f) = \C(-,g) \C(-,h)$ in $\mod \C$, where $h \in \C(X,Z)$ is such that $\C(-,h)$ is an epimorphism in $\mod \C$. But as $\C(-,Z)$ is projective in $\mod \C$, we conclude that $\C(-,h)$ is a split epimorphism, hence so is $h$, by the Yoneda lemma. Moreover, from the previous exact sequence in $\mod \C$, we also get that $g$ is a monomorphism in $\C$. Therefore, by the Yoneda lemma, we obtain a decomposition $f = gh$ in $\C$, where $h$ is a split epimorphism and $g$ is a monomorphism. Conversely, any decomposition $f = gh$ in $\C$ with $h$ a split epimorphism and $g$ a monomorphism leads to a factorization $\C(-,f) = \C(-,g) \C(-,h)$ in $\mod \C$, where $\C(-,h)$ is an epimorphism and $\C(-,g)$ is a monomorphism. In this case, $\C(-,g)$ is the image of $\C(-,f)$ in $\mod \C$, so that it gives a projective resolution of $F$ in $\mod \C$ as mentioned above.

If $\C$ is left coherent and $\gldim (\mod \C^{\op}) \leqslant 1$, then a similar discussion can be carried out, by considering $\mod \C^{\op}$ instead of $\mod \C$, which should lead to the notion of a ``$0$-cokernel'' of a morphism in $\C$.

Motivated by the above paragraphs, given an arbitrary category $\B$ and a morphism $f$ in $\B$, we define a \textit{$0$-kernel} of $f$ to be a monomorphism $g$ in $\B$ for which there is a decomposition $f = gh$ in $\B$, where $h$ is a split epimorphism. Dually, a \textit{$0$-cokernel} of $f$ is an epimorphism $s$ in $\B$ for which there is a decomposition  $f = rs$ in $\B$, where $r$ is a split monomorphism. We say that $\B$ \textit{has $0$-kernels} if every morphism in $\B$ has a $0$-kernel, and that $\B$ \textit{has $0$-cokernels} if every morphism in $\B$ has a $0$-cokernel.

We can now describe the axioms ($0$F1) and ($0$F1$^{\op}$) of a $0$-abelian category in terms of $0$-kernels and $0$-cokernels, respectively. In fact, consider the following axioms for the category $\C$:

\begin{enumerate}[leftmargin=3.8em]
    \item[\textup{($0$A1)}] $\C$ has $0$-kernels.
    \item[\textup{($0$A1$^{\op}$)}] $\C$ has $0$-cokernels.
\end{enumerate}

\begin{proposition}\label{proposition.1}
The axioms \textup{($0$F1)} and \textup{($0$F1$^{\op}$)} are equivalent to \textup{($0$A1)} and \textup{($0$A1$^{\op}$)}, respectively.
\end{proposition}

\begin{proof}
We only prove that ($0$F1) and ($0$A1) are equivalent. Once this is done, we can deduce that ($0$F1$^{\op}$) and ($0$A1$^{\op}$) are equivalent by considering opposite categories.

If $\C$ satisfies the axiom ($0$F1), then we have already seen in one of the previous paragraphs that every morphism in $\C$ has a $0$-kernel, so that $\C$ satisfies ($0$A1). Thus, it remains to prove the converse. In what follows, assume that $\C$ satisfies ($0$A1).

First, we will verify that $\C$ is right coherent. By \cite[Proposition 3.1]{2409.10438}, it suffices to show that $\C$ has weak kernels. Well, let $f \in \C(X,Y)$ be an arbitrary morphism in $\C$. Given that $\C$ has $0$-kernels, there is a decomposition $f = gh$ in $\C$, where $h$ is a split epimorphism and $g$ is a monomorphism. Let $h'$ be a morphism in $\C$ such that $hh' = 1$, and take a decomposition $1 - h'h = ab$ in $\C$, where $b$ is a split epimorphism and $a$ is a monomorphism. We claim that $a$ is a kernel of $f$. Indeed, $fab = gh(1 - h'h) = 0$, which implies that $fa = 0$ since $b$ is a split epimorphism. Moreover, if $w$ is a morphism in $\C$ for which $fw = 0$, then $ghw = 0$, which implies that $hw = 0$ as $g$ is a monomorphism. Thus, from $1 - h'h = ab$, we obtain that $w = abw$, so that $w$ factors through $a$ (in a unique way since $a$ is a monomorphism). Therefore, $a$ is a kernel of $f$.\footnote{Also, note that $a$ is a split monomorphism. In fact, because $1- h'h$ is idempotent and $1 - h'h = ab$, we have $abab = ab$. But $b$ is a split epimorphism and $a$ is a monomorphism, hence $ba = 1$.} Consequently, every morphism in $\C$ has a kernel, and hence a weak kernel.

Next, we will show that $\gldim (\mod \C) \leqslant 1$. For this, let $F \in \mod \C$ be arbitrary, and take a projective presentation \[ \begin{tikzcd}
{\C(-,X)} \arrow[r, "{\C(-,f)}"] &[1.2em] {\C(-,Y)} \arrow[r] & F \arrow[r] & 0
\end{tikzcd} \] of $F$ in $\mod \C$ with $f \in \C(X,Y)$. Since $\C$ has $0$-kernels, there is a decomposition $f = gh$ in $\C$, where $h$ is a split epimorphism and $g$ is a monomorphism. As it was remarked earlier, this leads to a decomposition $\C(-,f) = \C(-,g) \C(-,h)$ in $\mod \C$, where $\C(-,h)$ is an epimorphism and $\C(-,g)$ is a monomorphism, so that $\C(-,g)$ is the image of $\C(-,f)$. Thus, the kernel (object) of $\C(-,Y) \to F$ is projective, which implies that $\pdim F \leqslant 1$. Hence $\gldim (\mod \C) \leqslant 1$, and $\C$ satisfies the axiom ($0$F1).
\end{proof}

Observe that, in the proof of Proposition \ref{proposition.1}, we have also proved that if $\C$ has $0$-kernels, then $\C$ has kernels, and kernels in $\C$ are split monomorphisms. By duality, we also conclude that if $\C$ has $0$-cokernels, then $\C$ has cokernels, and cokernels in $\C$ are split epimorphisms. The converse of these statements are also true. Moreover, there is no need to assume that $\C$ is additive and idempotent complete to prove all these results. In fact, it suffices to assume that $\C$ is a preadditive category. Also, these results imply that $0$-abelian categories are precisely the split quasi-abelian categories, that is, quasi-abelian categories for which every short exact sequence splits. We will collect (and prove) these facts in a separate paper, \cite{0abelian}, which will be devoted to a comprehensive study of $0$-abelian categories without relying on the functorial approach.

Another basic fact that we can prove is that the existence of $0$-kernels or $0$-cokernels imply idempotent completeness, see Proposition \ref{proposition.2}. Therefore, we can characterize $0$-abelian categories as follows:

\begin{theorem}\label{theorem.3}
A category $\B$ is $0$-abelian if and only if $\B$ is additive and has both $0$-kernels and $0$-cokernels.
\end{theorem}

\begin{proof}
Follows from Propositions \ref{proposition.1} and \ref{proposition.2}.
\end{proof}

We can consider Theorem \ref{theorem.3} as the culmination of our efforts to understand what a $0$-abelian category is. The cherry on top of it is the fact that we can also describe $0$-kernels and $0$-cokernels in terms of universal properties, which are similar to the ones for kernels and cokernels, see Proposition \ref{proposition.10} and the paragraph proceeding its proof. In particular, we conclude that $0$-kernels and $0$-cokernels are unique up to isomorphism (although it is not difficult to prove their uniqueness directly from their definitions).

\subsection{An interlude}\label{subsection.2.2}
One could now shift perspectives and study $0$-abelian categories through their description given in Theorem \ref{theorem.3}. For instance, we can use it to conclude the following result, which says that the property of being $0$-abelian is ``hereditary''.

\begin{corollary}\label{corollary.1}
If $\C$ is a $0$-abelian category, then every subcategory of $\C$ that is closed under finite direct sums and direct summands is $0$-abelian.
\end{corollary}

\begin{proof}
Assume that $\C$ is a $0$-abelian category, and let $\B$ be a subcategory of $\C$ that is closed under finite direct sums and direct summands. It is clear that a $0$-kernel and a $0$-cokernel of a morphism in $\B$ taken in $\C$ is also, respectively, a $0$-kernel and a $0$-cokernel of it in $\B$. Therefore, it follows from Theorem \ref{theorem.3} that $\B$ is $0$-abelian.
\end{proof}

However, in order to keep consistency, we will explore the description of $0$-abelian categories given in Theorem \ref{theorem.3} in the separate paper \cite{0abelian} that we have already mentioned. For now, we will continue to follow the functorial approach. In particular, let us mention that, through this approach, van Roosmalen proved in \cite[Proposition 4.2]{MR2373322} a stronger converse of Corollary \ref{corollary.1}, which leads to the result below, that shows that $0$-abelianness is a ``local-to-global'' property.

\begin{proposition}\label{proposition.14}
Let $\C$ be an additive and idempotent complete category. Then $\C$ is $0$-abelian if and only if $\add X$ is $0$-abelian for every $X \in \C$.
\end{proposition}

\begin{proof}
Follows from Corollary \ref{corollary.1} and \cite[Proposition 4.2]{MR2373322}.
\end{proof}

We remark that Corollary \ref{corollary.1} and Proposition \ref{proposition.14} were also proved by Auslander and Reiten in \cite[Propositions 1.4 and 1.5]{MR404240} for the case where there is a commutative artinian ring $R$ such that $\C$ is $R$-linear and $\C(X,Y)$ is finitely generated as an $R$-module for every $X,Y \in \C$. Moreover, in our context, \cite[Theorem 1.6]{MR404240} becomes the following:

\begin{corollary}\label{corollary.2}
Let $\C$ be an additive and idempotent complete category. Then $\C$ is $0$-abelian if and only if the endomorphism ring of every object in $\C$ is semi-hereditary.\footnote{See Section \ref{section.5} for the definition of a ``semi-hereditary'' ring.}
\end{corollary}

\begin{proof}
For each $X \in \C$, the categories $\add X$ and $\proj \End(X)$ are equivalent, where $\End(X)$ denotes the endomorphism ring of $X$ in $\C$, see \cite[Section 8]{2409.10438}. Therefore, the result follows from Proposition \ref{proposition.14} and Theorem \ref{theorem.7}.
\end{proof}

More results like Corollary \ref{corollary.2}, which relate $0$-abelian categories and ``semi-hereditary'' rings, will be discussed in Section \ref{section.5}. For the moment, we move on to extend some basic results on $n$-abelian categories for $n$ positive to the case $n = 0$.

\subsection{The von Neumann regular case}\label{subsection.2.3}

As it was pointed out in \cite[Remark 3.4]{MR3519980} and \cite[Proposition 4.7]{2409.10438}, von Neumann regular categories are the categories that are $n$-abelian for every positive integer $n$. We can now extend this statement and say that these are the categories that are $n$-abelian for every nonnegative integer $n$. Let us explain how we can conclude this using the functorial approach.

Recall from \cite[Proposition 4.1]{2409.10438} that $\C$ is von Neumann regular if and only if $\C$ is right coherent and $\gldim (\mod \C) = 0$, which is the case if and only if $\C$ is left coherent and $\gldim (\mod \C^{\op}) = 0$. Thus, obviously, every von Neumann regular category is $0$-abelian, and if $\C$ is a $0$-abelian category which is not von Neumann regular, then $\gldim (\mod \C) = \gldim (\mod \C^{\op}) = 1$.\footnote{Note that this last statement is the case $n = 0$ of \cite[Corollary 4.6]{2409.10438}.} Therefore, we can easily adapt the proof and the statement of \cite[Proposition 4.7]{2409.10438} to obtain the result below.

\begin{proposition}
Let $\C$ be an additive and idempotent complete category. The following are equivalent:
\begin{enumerate}
    \item[(a)] $\C$ is von Neumann regular.
    \item[(b)] $\C$ is $n$-abelian for every nonnegative integer $n$.
    \item[(c)] There are two distinct nonnegative integers $m$ and $n$ for which $\C$ is $m$-abelian and $n$-abelian.
\end{enumerate}
\end{proposition}

\begin{proof}
Follows from the proof of \cite[Proposition 4.7]{2409.10438} and the previous comments.
\end{proof}

Another way to characterize von Neumann regular categories is to say that they are the $n$-abelian categories whose $n$-exact sequences split. This result was essentially first proved in \cite[Theorem 3.9]{MR3519980} for the case when $n$ is a positive integer, and later proved in \cite[Proposition C.4]{2409.10438} through functorial methods. Now, we can also prove it for the case $n = 0$, as we show in Proposition \ref{proposition.11}. To do this, we first need to explain what a ``$0$-exact sequence'' is, which we do in the next subsection.

\subsection{\texorpdfstring{$0$}{0}-Exact sequences}

Recall from \cite[Definition 2.4]{MR3519980} that an \textit{$n$-exact sequence} in $\C$ is a sequence \[ \begin{tikzcd}
Z_{n+1} \arrow[r, "h_{n+1}"] &[0.7em] Z_{n} \arrow[r, "h_{n}"] &[0.05em] \cdots \arrow[r, "h_{2}"] & Z_{1} \arrow[r, "h_{1}"] & Z_{0}
\end{tikzcd} \] of morphisms in $\C$ such that {\relsize{-0.5} \[ \begin{tikzcd}
0 \arrow[r] & {\C(-,Z_{n+1})} \arrow[r, "{\C(-,h_{n+1})}"] &[2.7em] {\C(-,Z_{n})} \arrow[r, "{\C(-,h_{n})}"] &[1.85em] \cdots \arrow[r, "{\C(-,h_{2})}"] &[1.8em] {\C(-,Z_{1})} \arrow[r, "{\C(-,h_{1})}"] &[1.8em] {\C(-,Z_{0})}
\end{tikzcd} \]} \hspace{-0.667em} and {\relsize{-0.5} \[ \begin{tikzcd}
0 \arrow[r] & {\C(Z_{0},-)} \arrow[r, "{\C(h_{1},-)}"] &[1.8em] {\C(Z_{1},-)} \arrow[r, "{\C(h_{2},-)}"] &[1.8em] \cdots \arrow[r, "{\C(h_{n},-)}"] &[1.85em] {\C(Z_{n},-)} \arrow[r, "{\C(h_{n+1},-)}"] &[2.7em] {\C(Z_{n+1},-)}
\end{tikzcd} \]} \hspace{-0.57em} are exact sequences in $\Mod \C$ and in $\Mod \C^{\op}$, respectively. Although $n$-exact sequences were originally defined for a positive integer $n$, we can also consider them when $n = 0$. By doing so, we see that a \textit{$0$-exact sequence} in $\C$ is a morphism in $\C$ which is both a monomorphism and an epimorphism, that is, a \textit{bimorphism}.

Observe that we might consider the definition of a $0$-exact sequence in an arbitrary category, since it is nothing but a bimorphism. Also, given that the term ``bimorphism'' is already well established, we will preferably use it to refer to a $0$-exact sequence.

Following \cite[Appendix C]{2409.10438}, we say that a bimorphism \textit{splits} if it is both a split monomorphism and a split epimorphism, that is, if it is an isomorphism. Clearly, a bimorphism splits if and only if it is a split epimorphism if and only if it is a split monomorphism, and these facts constitute the case $n = 0$ of \cite[Proposition C.2]{2409.10438}. Furthermore, observe that a category has the property that all of its bimorphisms split if and only if the category is balanced.

Now that we have provided the above definitions and remarks, we could aim to extend results on $n$-abelian categories and their $n$-exact sequences to the case $n = 0$. With this in mind, it is worth recalling from \cite[Propositions C.1 and C.3]{2409.10438} that if $\C$ is an $n$-abelian category with $n$ positive, and if $m$ is a nonnegative integer which is different from $n$, then every $m$-exact sequence in $\C$ splits. It turns out that the same result holds when $n = 0$, and we state it below. Therefore, when dealing with $n$-abelian categories and $m$-exact sequences in it, one should indeed focus their attention to the case $m = n$, even if $n = 0$.

\begin{proposition}
If $\C$ is a $0$-abelian category and $m$ is a positive integer, then every $m$-exact sequence in $\C$ splits.
\end{proposition}

\begin{proof}
The proof of \cite[Proposition C.3]{2409.10438} carries over.
\end{proof}

Next, we prove the result that was mentioned in the end of Subsection \ref{subsection.2.3} for $n = 0$.

\begin{proposition}\label{proposition.11}
Let $\C$ be an additive and idempotent complete category. Then $\C$ is $0$-abelian and every bimorphism in $\C$ splits if and only if $\C$ is von Neumann regular.
\end{proposition}


\begin{proof}
Suppose that $\C$ is $0$-abelian and that every bimorphism in $\C$ splits. Let $F \in \mod \C$ be arbitrary, and take a projective resolution \[ \begin{tikzcd}
0 \arrow[r] & {\C(-,X)} \arrow[r, "{\C(-,f)}"] &[1.2em] {\C(-,Y)} \arrow[r] & F \arrow[r] & 0
\end{tikzcd} \] of $F$ in $\mod \C$, where $f \in \C(X,Y)$. We know from Theorem \ref{theorem.3} that $f$ has a $0$-cokernel, hence there is a decomposition $f = rs$ in $\C$, where $s$ is an epimorphism and $r$ is a split monomorphism. But note that $f$ is a monomorphism, hence so is $s$. Thus, $s$ is a bimorphism, which must split. Consequently, $f$ is a split monomorphism, which implies that $\C(-,f)$ is a split monomorphism, so that the above short exact sequence in $\mod \C$ splits. Therefore, $F$ is projective, and we conclude that $\gldim (\mod \C) = 0$.\footnote{We can also use Proposition \ref{proposition.13} to show that $F$ is projective. Indeed, we get that $F \simeq \mr(g) \oplus \C(-,Z)$ for some bimorphism $g$ in $\C$ and some $Z \in \C$. But $g$ splits, hence $\mr(g) = 0$ and $F \simeq \C(-,Z)$.} By \cite[Proposition 4.1]{2409.10438}, $\C$ is von Neumann regular.

Next, assume that $\C$ is von Neumann regular. It follows from \cite[Proposition 4.1]{2409.10438} that $\C$ is $0$-abelian with $\gldim (\mod \C) = 0$. Moreover, given a bimorphism $f \in \C(X,Y)$, we can consider a short exact sequence in $\mod \C$ as above by letting $F$ be the cokernel (object) of $\C(-,f)$. However, in this case, $F$ is projective, so that this short exact sequence splits. Therefore, $\C(-,f)$ is a split monomorphism, hence so is $f$, by the Yoneda lemma. Consequently, the bimorphism $f$ splits.
\end{proof}

In other words, Proposition \ref{proposition.11} says that a category is $0$-abelian and balanced if and only if it is von Neumann regular.

More results about bimorphisms in $0$-abelian categories will be discussed in \mbox{Section \ref{section.4}.} But before that, let us examine the nature of $0$-kernels and $0$-cokernels in Section \ref{section.3}, specifically, the factorizations that they induce on a given morphism.


\section{Factorizations of morphisms}\label{section.3}

We know from Theorem \ref{theorem.3} that a $0$-abelian category is an additive category for which all of its morphisms can be decomposed both as a split epimorphism followed by a monomorphism and as an epimorphism followed by a split monomorphism. In this section, we investigate how these decompositions compare with the factorization of a morphism in an abelian category as an epimorphism followed by a monomorphism. Moreover, we explain how the cases $n = 0$ and $n = 1$ might help to understand the existence of similar factorizations of morphisms in an $n$-abelian category when $n \geqslant 2$. Since our approach is functorial, we begin by pointing out how to relate epimorphisms and monomorphisms in $\C$ to objects in the category $\mod \C$.

Recall that $F \in \mod \C$ is called \textit{effaceable} if $F^{\ast} = 0$, that is, $\Hom(F,\C(-,X)) = 0$ for all $X \in \C$. The following result is folklore, and goes back to \cite[Proposition 3.2]{MR0212070}.

\begin{proposition}\label{proposition.6}
Let $F \in \mod \C$. The following are equivalent:
\begin{enumerate}
    \item[(a)] If $f$ is a morphism in $\C$ such that $F \simeq \mr(f)$, then $f$ is an epimorphism.
    \item[(b)] There is an epimorphism $f$ in $\C$ such that $F \simeq \mr(f)$.
    \item[(c)] $F^{\ast} = 0$.
\end{enumerate}
\end{proposition}

\begin{proof}
Trivially, (a) implies (b). Below, we show that (b) implies (c), and one can verify that (c) implies (a) by using a similar argument.

Suppose that there is an epimorphism $f \in \C(X,Y)$ in $\C$ such that $F \simeq \mr(f)$, so that there is an exact sequence \[ \begin{tikzcd}
{\C(-,X)} \arrow[r, "{\C(-,f)}"] &[1.2em] {\C(-,Y)} \arrow[r] & F \arrow[r] & 0
\end{tikzcd} \] in $\Mod \C$. By applying the functor $(-)^{\ast}$ to this sequence, we obtain an exact sequence \[ \begin{tikzcd}
0 \arrow[r] & F^{\ast} \arrow[r] & {\C(Y,-)} \arrow[r, "{\C(f,-)}"] &[1.2em] {\C(X,-)}
\end{tikzcd} \] in $\Mod \C^{\op}$. Since $f$ is an epimorphism, it follows that $F^{\ast} = 0$.
\end{proof}

While Proposition \ref{proposition.6} deals with epimorphisms in $\C$, its corresponding statement for monomorphisms in $\C$ is slightly different, and we present it below. This is a well known result, see, for example, \cite[Proposition 2.3]{MR4575371}.

\begin{proposition}\label{proposition.12}
Assume that $\C$ is right coherent, and let $F \in \mod \C$. The following are equivalent:
\begin{enumerate}
    \item[(a)] If $f$ is a morphism in $\C$ such that $F \simeq \mr(f)$, then there is a decomposition $f = gh$ in $\C$, where $h$ is a split epimorphism and $g$ is a monomorphism.
    \item[(b)] There is a monomorphism $f$ in $\C$ such that $F \simeq \mr(f)$.
    \item[(c)] $\pdim F \leqslant 1$.
\end{enumerate}
\end{proposition}

\begin{proof}
It is straightforward to check that (b) is equivalent to (c), and it follows from our discussion on $0$-kernels in Subsection \ref{subsection.2.1} that (a) is equivalent to (c).
\end{proof}

We denote the subcategory of $\mod \C$ consisting of the effaceable objects by $\eff \C$. Moreover, when $\C$ is right coherent, we let $\modone \C$ be the subcategory of $\mod \C$ consisting of the objects $F \in \mod \C$ with $\pdim F \leqslant 1$. With these subcategories, we can describe when every morphism in $\C$ factors as an epimorphism followed by a monomorphism.

\begin{proposition}\label{proposition.7}
Assume that $\C$ is right coherent. Then every morphism $f$ in $\C$ decomposes as $f = gh$ in $\C$ with $h$ an epimorphism and $g$ a monomorphism if and only if for every $F \in \mod \C$ there is a short exact sequence \[ \begin{tikzcd}
0 \arrow[r] & F_{\mathsf{e}} \arrow[r] & F \arrow[r] & F_{1} \arrow[r] & 0
\end{tikzcd} \] in $\mod \C$ with $F_{\mathsf{e}} \in \eff \C$ and $F_{1} \in \modone \C$.
\end{proposition}

\begin{proof}
Suppose that every morphism in $\C$ decomposes as an epimorphism followed by a monomorphism. Let $F \in \mod \C$ be arbitrary, and take a projective presentation \[ \begin{tikzcd}
{\C(-,X)} \arrow[r, "{\C(-,f)}"] &[1.2em] {\C(-,Y)} \arrow[r] & F \arrow[r] & 0
\end{tikzcd} \] of $F$ in $\mod \C$, where $f \in \C(X,Y)$. By assumption, there is an epimorphism $h \in \C(X,Z)$ and a monomorphism $g \in \C(Z,Y)$ in $\C$ for which $f = gh$. Consequently, we have $\C(-,f) = \C(-,g) \C(-,h)$ in $\mod \C$ with $\C(-,g)$ a monomorphism. Then, by considering the circular sequence for the composition $\C(-,g) \C(-,h)$, we obtain a short exact sequence \[ \begin{tikzcd}
0 \arrow[r] & F_{\mathsf{e}} \arrow[r] & F \arrow[r] & F_{1} \arrow[r] & 0
\end{tikzcd} \] in $\mod \C$, where $F_{\mathsf{e}} = \mr(h)$ and $F_{1} = \mr(g)$.\footnote{It is also possible to use the snake lemma or to proceed as in the proof of \cite[Corollary 2.5]{MR4575371} to deduce the existence of such a short exact sequence. However, the use of the circular sequence (which was suggested by Alex Martsinkovsky) gives a much more elegant and direct argument for its existence.} By Propositions \ref{proposition.6} and \ref{proposition.12}, we have $F_{\mathsf{e}} \in \eff \C$ and $F_{1} \in \modone \C$.

Conversely, assume that every $F \in \mod \C$ fits into a short exact sequence \[ \begin{tikzcd}
0 \arrow[r] & F_{\mathsf{e}} \arrow[r] & F \arrow[r] & F_{1} \arrow[r] & 0
\end{tikzcd} \] in $\mod \C$ with $F_{\mathsf{e}} \in \eff \C$ and $F_{1} \in \modone  \C$. Let $f \in \C(X,Y)$ be a morphism in $\C$, and denote $\mr(f) = F$, so that there is an exact sequence \[ \begin{tikzcd}
{\C(-,X)} \arrow[r, "{\C(-,f)}"] &[1.2em] {\C(-,Y)} \arrow[r] & F \arrow[r] & 0
\end{tikzcd} \] in $\mod \C$, and also a short exact sequence in $\mod \C$ as above, with $F_{\mathsf{e}} \in \eff \C$ and $F_{1} \in \modone \C$. By taking the composition of $\C(-,Y) \to F$ and $F \to F_{1}$, we get an epimorphism $\C(-,Y) \to F_{1}$ in $\mod \C$ whose kernel (object) is projective, as $\pdim F_{1} \leqslant 1$, hence it is isomorphic to $\C(-,Z)$ for some $Z \in \C$. Therefore, we conclude the existence of a commutative diagram with exact rows \[ \begin{tikzcd}
            & {\C(-,X)} \arrow[r, "{\C(-,f)}"] \arrow[d, "{\C(-,h)}"'] &[1.2em] {\C(-,Y)} \arrow[r] \arrow[d, equal] & F \arrow[r] \arrow[d] & 0 \\
0 \arrow[r] & {\C(-,Z)} \arrow[r, "{\C(-,g)}"']                        & {\C(-,Y)} \arrow[r]                  & F_{1} \arrow[r]       & 0
\end{tikzcd} \] in $\mod \C$, where $h \in \C(X,Z)$ and $g \in \C(Z,Y)$. In particular, $g$ is a monomorphism, and it follows from the Yoneda lemma that $f = gh$. Moreover, we can use either the circular sequence or the snake lemma to deduce that $F_{\mathsf{e}} \simeq \mr(h)$. Thus, we conclude from Proposition \ref{proposition.6} that $h$ is an epimorphism.
\end{proof}

One could go one step further in Proposition \ref{proposition.7} and ask what it means in terms of $\eff \C$ and $\modone \C$ to say that every morphism in $\C$ decomposes as an epimorphism followed by a monomorphism in a ``unique way''. In order to make sense of this question, we recall below the notion of a ``factorization system'', see \cite{MR1291599} and \cite{MR322004} for more details.

Let $\B$ be an arbitrary category, and let $f \in \B(X,Y)$ and $g \in \B(Z,W)$ be morphisms in $\B$. We say that $f$ and $g$ are \textit{orthogonal} if for every commutative square \[ \begin{tikzcd}
X \arrow[r, "f"] \arrow[d, "x"'] & Y \arrow[d, "y"] \\
Z \arrow[r, "g"']                & W               
\end{tikzcd} \] in $\B$, there is a unique morphism $r \in \B(Y,Z)$ such that $y = gr$ and $x = rf$. We use the notation $f \perp g$ to indicate that $f$ and $g$ are orthogonal. A \textit{factorization system} in $\B$ is a pair $(\mathcal{E}, \mathcal{M})$ of subcategories $\mathcal{E}$ and $\mathcal{M}$ of $\Mor \B$, the category of morphisms in $\B$, satisfying the following conditions:
 \begin{enumerate}
     \item[(a)] $\mathcal{E}$ and $\mathcal{M}$ are closed under isomorphisms (in the category $\Mor \B$).
     \item[(b)] Every morphism $f$ in $\B$ can be written as $f = me$, where $e \in \mathcal{E}$ and $m \in \mathcal{M}$.
     \item[(c)] $\mathcal{E} \perp \mathcal{M}$, that is, $e \perp m$ for every $e \in \mathcal{E}$ and $m \in \mathcal{M}$.
 \end{enumerate}
It is not difficult to see that if $(\mathcal{E},\mathcal{M})$ is a factorization system in $\B$, then the decomposition of a morphism $f$ in $\B$ given by item (b) is unique up to isomorphism.

Now, let $\Epi$ and $\Mono$ be the subcategories of $\Mor \C$ consisting of, respectively, the epimorphisms and monomorphisms in $\C$. Made precise, the problem that we had sketched earlier is on how to characterize when $(\Epi, \Mono)$ is a factorization system in $\C$ in terms of $\eff \C$ and $\modone \C$. We answer this question in the next proposition, but first, we need to state one more definition.

Recall from \cite{MR191935} that a \textit{torsion pair} in an abelian category $\A$ is a pair $(\mathcal{T}, \mathcal{L})$ of subcategories $\mathcal{T}$ and $\mathcal{L}$ of $\A$ satisfying the following conditions:
\begin{enumerate}
    \item[(a)] $\mathcal{T}$ and $\mathcal{L}$ are closed under isomorphisms (in the category $\A$).
    \item[(b)] Every object $X \in \A$ fits into a short exact sequence \[ \begin{tikzcd}
0 \arrow[r] & T \arrow[r] & X \arrow[r] & L \arrow[r] & 0
\end{tikzcd} \] in $\A$ with $T \in \mathcal{T}$ and $L \in \mathcal{L}$.
    \item[(c)] $\A(\mathcal{T},\mathcal{L}) = 0$, that is, $\A(T,L) = 0$ for every $T \in \mathcal{T}$ and $L \in \mathcal{L}$.
\end{enumerate}
It is worth mentioning that if $(\mathcal{T}, \mathcal{L})$ is a torsion pair in $\A$, then the short exact sequence associated to an object $X \in \A$ in item (b) is unique up to isomorphism.
 
\begin{proposition}\label{proposition.8}
Assume that $\C$ is right coherent. Then $(\Epi, \Mono)$ is a factorization system in $\C$ if and only if $(\eff \C, \modone \C)$ is a torsion pair in $\mod \C$.
\end{proposition}

\begin{proof}
Suppose that $(\Epi, \Mono)$ is a factorization system in $\C$. It is clear that $\eff \C$ and $\modone \C$ are closed under isomorphisms. Thus, due to Proposition \ref{proposition.7}, it suffices to verify that $\Hom(\eff \C, \modone \C) = 0$ to conclude that $(\eff \C, \modone \C)$ is a torsion pair in $\mod \C$. Well, let $F \in \eff \C$ and $G \in \modone \C$, and take $\alpha \in \Hom(F,G)$. By Propositions \ref{proposition.6} and \ref{proposition.12}, there is an epimorphism $f \in \C(X,Y)$ in $\C$ such that $F \simeq \mr(f)$, and there is a monomorphism $g \in \C(Z,W)$ in $\C$ for which $G \simeq \mr(g)$. Consequently, there is a commutative diagram with exact rows \[ \begin{tikzcd}
            & {\C(-,X)} \arrow[r, "{\C(-,f)}"] \arrow[d, "{\C(-,x)}"'] &[1.2em] {\C(-,Y)} \arrow[r] \arrow[d, "{\C(-,y)}"] & F \arrow[r] \arrow[d, "\alpha"] & 0 \\
0 \arrow[r] & {\C(-,Z)} \arrow[r, "{\C(-,g)}"']                        & {\C(-,W)} \arrow[r]                        & G \arrow[r]                     & 0
\end{tikzcd} \] in $\mod \C$ for some $x \in \C(X,Z)$ and $y \in \C(Y,W)$. Then, by the Yoneda lemma, we have $yf = gx$, and because $f \perp g$, there is a (unique) morphism $r \in \C(Y,Z)$ for which $y = gr$ and $x = rf$. Therefore, $\C(-,y) = \C(-,g) \C(-,r)$, and we can conclude from the above diagram that $\alpha = 0$.

For the converse, assume that $(\eff \C, \modone \C)$ is a torsion pair in $\mod \C$. Clearly, $\Epi$ and $\Mono$ are closed under isomorphisms in $\Mor \C$, and it follows from Proposition \ref{proposition.7} that every morphism in $\C$ factors as an epimorphism followed by a monomorphism. Therefore, in order to conclude that $(\Epi , \Mono)$ is a factorization system in $\C$, it only remains to show that $\Epi \hspace{-0.1em} \perp \hspace{-0.1em} \Mono$. So, let $f \in \C(X,Y)$ and $g \in \C(Z,W)$ be, respectively, an epimorphism and a monomorphism in $\C$, and suppose that there is a commutative diagram \[ \begin{tikzcd}
X \arrow[r, "f"] \arrow[d, "x"'] & Y \arrow[d, "y"] \\
Z \arrow[r, "g"']                & W               
\end{tikzcd} \] in $\C$. By applying the Yoneda embedding to this diagram and by taking the cokernels of $\C(-,f)$ and $\C(-,g)$, we obtain a commutative diagram with exact rows in $\mod \C$ as the one we had above, with $F \simeq \mr(f)$ and $G \simeq \mr(g)$. However, we know from Propositions \ref{proposition.6} and \ref{proposition.12} that $F \in \eff \C$ and $G \in \modone \C$. Thus, $\Hom(F,G) = 0$, which implies (following the diagram in $\mod \C$) that $\C(-,y)$ factors uniquely through $\C(-,g)$. By the Yoneda lemma, this means that there is a unique morphism $r \in \C(Y,Z)$ such that $\C(-,y) = \C(-,g) \C(-,r)$, which is equivalent to $y = gr$. Furthermore, by using the fact that $g$ is a monomorphism, we can deduce that $x = rf$, hence $f \perp g$.
\end{proof}

As the reader might have noticed, Propositions \ref{proposition.7} and \ref{proposition.8} seem to admit more general statements, and the way we have defined a factorization system was intentionally chosen to highlight its similarities with the notion of a torsion pair. Nevertheless, we do not attempt to give a general theory on the relationship between factorization systems in $\C$ and torsion pairs in $\mod \C$, as it would fall out of the scope of this paper. Still, the reader should keep this observation in mind as we proceed.

It is a standard fact that if $\C$ is abelian, then $(\Epi, \Mono)$ is a factorization system in $\C$. Therefore, we deduce from Proposition \ref{proposition.8} that if $\C$ is abelian, then $(\eff \C, \modone \C)$ is a torsion pair in $\mod \C$. However, it is beneficial to prove this result using functorial methods, as in \cite[Proposition 3.6]{MR4392222}, \cite[Proposition 7.5]{MR4575371} and \cite[Proposition 3]{MR2471947}. Below, we give yet another such proof.

\begin{proposition}\label{proposition.9}
If $\C$ is an abelian category, then $(\eff \C, \modone \C)$ is a torsion pair in $\mod \C$.
\end{proposition}

\begin{proof}
To begin with, note that, by Theorem \ref{theorem.1}, $\C$ is right coherent and every object in $\modone \C$ is a syzygy in $\mod \C$. Moreover, it is clear that every morphism in an abelian category can be written as an epimorphism followed by a monomorphism. Thus, in view of Proposition \ref{proposition.7}, it only remains to show that $\Hom(\eff \C, \modone \C) = 0$ to conclude that $(\eff \C, \modone \C)$ is a torsion pair in $\mod \C$.

Let $F \in \eff \C$ and $G \in \modone \C$, and consider a morphism $\alpha \in \Hom(F,G)$. Then, following the proof of Proposition \ref{proposition.8}, there is a commutative diagram with exact rows \[ \begin{tikzcd}
            & {\C(-,X)} \arrow[r, "{\C(-,f)}"] \arrow[d, "{\C(-,x)}"'] &[1.2em] {\C(-,Y)} \arrow[r, "\beta"] \arrow[d, "{\C(-,y)}"] & F \arrow[r] \arrow[d, "\alpha"] & 0 \\
0 \arrow[r] & {\C(-,Z)} \arrow[r, "{\C(-,g)}"']                        & {\C(-,W)} \arrow[r, "\gamma"']                        & G \arrow[r]                     & 0
\end{tikzcd} \] in $\mod \C$, where $f \in \C(X,Y)$, $g \in \C(Z,W)$, $x \in \C(X,Z)$ and $y \in \C(Y,W)$, with $f$ an epimorphism and $g$ a monomorphism. Also, because $G \in \modone \C$, it is a syzygy, so that there is a monomorphism $\delta \in \Hom (G,\C(-,V))$ in $\mod \C$ for some $V \in \C$. Now, note that $\alpha = 0$ if and only if $\delta \gamma \C(-,y) = 0$. Moreover, it follows from the Yoneda lemma that $\delta \gamma = \C(-,h)$ for some $h \in \C(W,V)$, hence we obtain that $\alpha = 0$ if and only if $hy = 0$. But, in fact, we have $\C(-,h)\C(-,y)\C(-,f) = 0$, which implies that $hyf = 0$, and then $hy = 0$, as $f$ is an epimorphism. Therefore, $\alpha = 0$.
\end{proof}

The results obtained in Propositions \ref{proposition.7}, \ref{proposition.8} and \ref{proposition.9} provide a useful perspective on the factorization of a morphism in an abelian category as an epimorphism followed by a monomorphism. Indeed, we can now use this perspective to understand how this factorization compares to the factorizations of a morphism in a $0$-abelian category as a split epimorphism followed by a monomorphism and as an epimorphism followed by a split monomorphism. To realize this comparison, we present below the analogues of Propositions \ref{proposition.7}, \ref{proposition.8} and \ref{proposition.9} for the case $n = 0$.

In what follows, let $\SplitEpi$ and $\SplitMono$ be the subcategories of $\Mor \C$ consisting of, respectively, the split epimorphisms and split monomorphisms in $\C$.

\begin{proposition}\label{proposition.13}
Assume that $\C$ is right coherent. The following are equivalent:
\begin{enumerate}
    \item[(a)] Every morphism $f$ in $\C$ decomposes as $f = rs$ in $\C$ with $s$ an epimorphism and $r$ a split monomorphism.\footnote{That is, $\C$ has $0$-cokernels.}
    \item[(b)] For every $F \in \mod \C$ there is a short exact sequence \[ \begin{tikzcd}
0 \arrow[r] & F_{\mathsf{e}} \arrow[r] & F \arrow[r] & \C(-,Z) \arrow[r] & 0
\end{tikzcd} \] in $\mod \C$ with $F_{\mathsf{e}} \in \eff \C$ and $Z \in \C$.
    \item[(c)] $(\Epi, \SplitMono)$ is a factorization system in $\C$.
    \item[(d)] $(\eff \C, \proj \C)$ is a torsion pair in $\mod \C$.
    \item[(e)] The category $\C$ is $0$-abelian.
    \item[(f)] Every $F \in \mod \C$ is of the form $F \simeq \mr(f) \oplus \C(-,Z)$ for some bimorphism $f$ in $\C$ and some $Z \in \C$.
\end{enumerate}
\end{proposition}

\begin{proof}
It is clear that (a) is equivalent to (c), and that (b) is equivalent to (d). Also, the proof of Proposition \ref{proposition.7} carries over to show that (a) is equivalent to (b). Furthermore, it follows from Proposition \ref{proposition.1} and \cite[Theorem A.5]{2409.10438} that (a) is equivalent to (e). Finally, it is easy to use Propositions \ref{proposition.6} and \ref{proposition.12} to deduce that conditions (b) and (e) together imply (f), and by Proposition \ref{proposition.6}, (f) implies (b).
\end{proof}

We leave it to the reader to state the result that is dual to Proposition \ref{proposition.13}, which is obtained by replacing $\C$ by $\C^{\op}$. We just remark that, in particular, it says that if $\C$ is left coherent, then $(\SplitEpi, \Mono)$ is a factorization system in $\C$ if and only if $(\eff \C^{\op}, \proj \C^{\op})$ is a torsion pair in $\mod \C^{\op}$ if and only if $\C$ is a $0$-abelian category.

At this moment, it is convenient to note that, by Proposition \ref{proposition.9}, if $\C$ is abelian, then $(\eff \C^{\op}, \modone \C^{\op})$ is a torsion pair in $\mod \C^{\op}$. Thus, one might ask what factorization system in $\C$ is obtained from this torsion pair in $\mod \C^{\op}$. Well, by taking $\C^{\op}$ in place of $\C$ in Proposition \ref{proposition.8}, we conclude that if $\C$ is left coherent, then $(\eff \C^{\op}, \modone \C^{\op})$ is a torsion pair in $\mod \C^{\op}$ if and only if $(\Epi, \Mono)$ is a factorization system in $\C$. So, when $\C$ is abelian, both torsion pairs $(\eff \C, \modone \C)$ in $\mod \C$ and $(\eff \C^{\op}, \modone \C^{\op})$ in $\mod \C^{\op}$ lead to the same factorization system in $\C$, namely, $(\Epi, \Mono)$.\footnote{Roughly speaking, this seems to happen due to the ordered pairs of integers $(1,n)$ and $(n,1)$ being the same when $n = 1$. For $n = 0$, we get the distinct pairs $(1,0)$ and $(0,1)$, each of which corresponds to a different factorization system in a $0$-abelian category.}

We can now conclude that the factorizations that we have been discussing in $0$-abelian categories and in abelian categories are analogous in the sense that $(\eff \C, \proj \C)$ and $(\eff \C^{\op}, \proj \C^{\op})$ are analogous to $(\eff \C, \modone \C)$ and $(\eff \C^{\op}, \modone \C^{\op})$, respectively. We could, perhaps, try to formalize these analogies, and attempt to develop a general theory for them, as we have hinted before, but the author believes that this would be more appropriate for a separate paper. Instead, let us point out how the ideas developed in this section could help to understand the existence of factorization systems in $n$-abelian categories, for $n \geqslant 2$, that are ``higher versions'' of $(\Epi, \SplitMono)$ and $(\SplitEpi, \Mono)$ when $n = 0$, and of $(\Epi, \Mono)$ when $n = 1$.

To begin with, let us make the trivial observation that, when $\C$ is right coherent, $\proj \C$ consists of the objects $F \in \mod \C$ with $\pdim F \leqslant 0$. With this in mind, it is natural to consider the subcategory $\mod_{n} \hspace{-0.1em} \C$ of $\mod \C$ consisting of the objects $F \in \mod \C$ with $\pdim F \leqslant n$, and to conjecture that $(\eff \C, \mod_{n} \hspace{-0.1em} \C)$ is a torsion pair in $\mod \C$ when $\C$ is $n$-abelian, since this is the case when $n = 0$ and $n = 1$, by Propositions \ref{proposition.13} and \ref{proposition.9}. Fortunately, this is also true for $n \geqslant 2$. Indeed, if $\C$ is $n$-abelian and $n \geqslant 2$, then it follows from Theorem \ref{theorem.1} that every $F \in \mod_{n} \hspace{-0.1em} \C$ is a syzygy in $\mod \C$, and we can proceed similarly to the proof of Proposition \ref{proposition.9} to show that $\Hom(\eff \C, \mod_{n} \hspace{-0.1em} \C) = 0$. Moreover, to deduce that every $F \in \mod \C$ fits into a short exact sequence \[ \begin{tikzcd}
0 \arrow[r] & F_{\mathsf{e}} \arrow[r] & F \arrow[r] & F_{n} \arrow[r] & 0
\end{tikzcd} \] in $\mod \C$ with $F_{\mathsf{e}} \in \eff \C$ and $F_{n} \in \mod_{n} \hspace{-0.1em} \C$, we can use the fact that $\C$ is equivalent to an \textit{$n$-cluster tilting subcategory} of an abelian category, see \cite[Corollary 1.3]{MR4301013} and \cite[Theorem A]{MR4514466}, and proceed as in the proof of \cite[Proposition 3.14]{EbrahimiNasr-Isfahani}. However, this argument for the existence of the above short exact sequence does not seem to be satisfactory, since it depends on regarding $\C$ as an $n$-cluster tilting subcategory of an ambient abelian category. Also, it does not give a way to construct projective presentations of $F_{\mathsf{e}}$ and $F_{n}$ based on a given projective presentation of $F$, in the same spirit of the proof of Proposition \ref{proposition.7}. Such a construction would be desirable and important to understand if the torsion pair $(\eff \C, \mod_{n} \hspace{-0.1em} \C)$ in $\mod \C$ gives rise to a factorization system in $\C$, so that the analogies that we have mentioned earlier would also hold for $n \geqslant 2$. In view of these remarks, we propose the following questions:
\begin{enumerate}
    \item[(Q1)] Is there a constructive proof for the fact that if $\C$ is $n$-abelian, then for every $F \in \mod \C$ there is a short exact sequence \[ \begin{tikzcd}
0 \arrow[r] & F_{\mathsf{e}} \arrow[r] & F \arrow[r] & F_{n} \arrow[r] & 0
\end{tikzcd} \] in $\mod \C$ with $F_{\mathsf{e}} \in \eff \C$ and $F_{n} \in \mod_{n} \hspace{-0.1em} \C$, which does not use the result that $\C$ is equivalent to an $n$-cluster tilting subcategory of an abelian category?
    \item[(Q2)] Is there a factorization system of the form $(\Epi, \mathcal{M})$ in an $n$-abelian category $\C$ that corresponds to the torsion pair $(\eff \C, \mod_{n} \hspace{-0.1em} \C)$ in $\mod \C$ as in the cases $n = 0$ and $n = 1$? If the answer is affirmative, what is $\mathcal{M}$?
\end{enumerate}
Note that, by replacing $\C$ by $\C^{\op}$ above, question (Q2) becomes whether there is a factorization system of the form $(\mathcal{E}, \Mono)$ in an $n$-abelian category $\C$ that corresponds to the torsion pair $(\eff \C^{\op}, \mod_{n} \hspace{-0.1em} \C^{\op})$ in $\mod \C^{\op}$.

The author hopes to have convinced the reader that a factorization system as the one proposed in question (Q2) should play the role of the factorization systems in $0$-abelian categories and abelian categories that we have discussed in this section. Finally, let us mention that one such candidate has been indicated in \cite[Proposition 3.13]{MR4860676} for the case when $\C$ is an $n$-cluster tilting subcategory of an ``abelian length category''. It would be interesting to examine how it relates to our desired factorization system, and whether it provides answers to questions (Q1) and (Q2).

\section{\texorpdfstring{$0$}{0}-Abelian categories with enough injectives or projectives}\label{section.4}

In this section, we investigate $0$-abelian categories with enough injectives or projectives through their categories of finitely presented modules. We begin by extending results on $n$-abelian categories from \cite[Section 7]{2409.10438} to the case $n = 0$, which turn out to be rather straightforward. Then we move on to define a bifunctor for a $0$-abelian category with enough injectives or projectives, which is analogous to the extension bifunctor for an abelian category. Once this bifunctor is defined, we give a few results that support this analogy. For instance, we present $0$-abelian versions of the long exact sequence for the extension bifunctors, of a conjecture due to Auslander on the direct summands of the extension functors, and of the Hilton-Rees theorem. We then end the section with applications to stable categories.

\subsection{Extending some previous results}

Recall from \cite[Section 7]{2409.10438} that the category $\C$ is called \textit{right comprehensive} when $\C(-,I)$ is an injective object in $\mod \C$ for every injective object $I \in \C$. Dually, $\C$ is called \textit{left comprehensive} when $\C(P,-)$ is injective in $\mod \C^{\op}$ for every projective object $P \in \C$. When $\C$ is both right and left comprehensive, we say that $\C$ is \textit{comprehensive}. It was essentially proved by Jasso in \cite[Theorem 3.12]{MR3519980} that every $n$-abelian category is comprehensive, whenever $n$ is a positive integer, see \cite[Proposition 7.3]{2409.10438}. We can now extend this result to the case $n = 0$:

\begin{proposition}\label{proposition.3}
Every $0$-abelian category is comprehensive.
\end{proposition}

\begin{proof}
Let $\C$ be a $0$-abelian category. Below, we show that $\C$ is right comprehensive. Then, by taking $\C^{\op}$ in place of $\C$, we can deduce that $\C$ is also left comprehensive.

Let $I \in \C$ be an injective object. In order to prove that $\C(-,I)$ is injective in $\mod \C$, we will show that $\Ext^{1}(-,\C(-,I)) = 0$. For this, let $F \in \mod \C$ be arbitrary, and take a projective resolution \[ \begin{tikzcd}
0 \arrow[r] & {\C(-,X)} \arrow[r, "{\C(-,f)}"] &[1.2em] {\C(-,Y)} \arrow[r] & F \arrow[r] & 0
\end{tikzcd} \] of $F$ in $\mod \C$, where $f \in \C(X,Y)$. By applying  $\Hom(-,\C(-,I))$ to this resolution, we obtain a complex in $\Ab$ which is, by the Yoneda lemma, isomorphic to \[ \begin{tikzcd}
0 \arrow[r] & {\Hom(F,\C(-,I))} \arrow[r] & {\C(Y,I)} \arrow[r, "{\C(f,I)}"] &[1.2em] {\C(X,I)} \arrow[r] & 0
\end{tikzcd}. \] Consequently, $\Ext^{1}(F,\C(-,I)) \simeq \C(X,I) / \Image \C(f,I)$. However, note that $f$ is a monomorphism, and because $I$ is injective, it follows that $\C(f,I)$ is an epimorphism. Therefore, $\Ext^{1}(F,\C(-,I)) = 0$, and we get that $\Ext^{1}(-,\C(-,I)) = 0$.
\end{proof}

Given that $0$-abelian categories are coherent and comprehensive, we obtain the result below, which is the case $n = 0$ of \cite[Proposition 7.6]{2409.10438}.

\begin{proposition}\label{proposition.4}
Assume that $\C$ is a $0$-abelian category. Then the following hold:
\begin{enumerate}
    \item[(a)] $\C$ has enough injectives if and only if $\domdim (\mod \C) \geqslant 1$.
    \item[(b)] $\C$ has enough projectives if and only if $\domdim (\mod \C^{\op}) \geqslant 1$.
\end{enumerate}
\end{proposition}

\begin{proof}
Follows from Proposition \ref{proposition.3} and \cite[Corollary 7.5]{2409.10438}.
\end{proof}

We can now state \cite[Theorem 7.14]{2409.10438} for $n = 0$, which is almost trivial due to Proposition \ref{proposition.4} and the way we have defined a $0$-abelian category.

\begin{theorem}\label{theorem.2}
Let $\C$ be an additive and idempotent complete category. Then $\C$ is $0$-abelian and has enough injectives if and only if $\C$ is coherent and $\gldim (\mod \C) \leqslant 1 \leqslant \domdim (\mod \C)$. Moreover, in this case, $\mod \C$ has enough injectives.
\end{theorem}

\begin{proof}
The main claim follows from Proposition \ref{proposition.4} and \cite[Theorem A.5]{2409.10438}. The second claim about $\mod \C$ having enough injectives can be verified by adapting the proof of \cite[Proposition 4.20]{MR4392222} to our context.
\end{proof}

The same arguments used above show that $\C$ is $0$-abelian and has enough projectives if and only if $\C$ is coherent and $\gldim (\mod \C^{\op}) \leqslant 1 \leqslant \domdim (\mod \C^{\op})$. Also, under such conditions, $\mod \C^{\op}$ has enough injectives. Moreover, although Theorem \ref{theorem.2} and its dual statement give inequalities for the global and dominant dimensions of $\mod \C$ and $\mod \C^{\op}$, there are only two possibilities for them. Indeed, if $\C$ is $0$-abelian and has enough injectives, then either $\C$ is von Neumann regular, so that $\gldim (\mod \C) = 0$ and $\domdim (\mod \C) = \infty$, or it happens that $\gldim (\mod \C) = 1 = \domdim (\mod \C)$. Dually, if $\C$ is $0$-abelian and has enough projectives, then either $\gldim (\mod \C^{\op}) = 0$ and $\domdim (\mod \C^{\op}) = \infty$, or $\gldim (\mod \C^{\op}) = 1 = \domdim (\mod \C^{\op})$.

Let us also remark that since we know from Theorem \ref{theorem.2} and its dual statement that $\mod \C$ and $\mod \C^{\op}$ have enough injectives when $\C$ is a $0$-abelian category with enough injectives and enough projectives, respectively, in these cases, it would be convenient to know how to get injective resolutions of objects in $\mod \C$ and in $\mod \C^{\op}$. We will show how to construct such injective resolutions in Proposition \ref{proposition.21} and its dual result. But before we get there, we need to develop the theory further.

\subsection{A bifunctor analogous to \texorpdfstring{$\Ext^{1}$}{Ext1}}\label{subsection.4.2}

In this subsection, we define a biadditive bifunctor $\E_{\C}(-,-)$ for a $0$-abelian category $\C$ with enough injectives or projectives that is analogous to the bifunctor $\Ext_{\A}^{1}(-,-)$ for an abelian category $\A$. Moreover, we present a few results concerning this bifunctor that provide evidence for this analogy. Among these results, we prove in Theorem \ref{theorem.4} that every bimorphism in a $0$-abelian category with enough injectives or projectives induces a ``not so long exact sequence'', which can be regarded as a $0$-abelian version of the long exact sequence involving the bifunctors $\Ext_{\A}^{i}(-,-)$ induced by a short exact sequence in an abelian category $\A$.

Let us start with some elementary observations. When one learns in kindergarten how to define the bifunctors $\Ext_{\A}^{i}(-,-)$ for an abelian category $\A$ with enough injectives or projectives, one uses injective or projective resolutions of objects in $\A$ to do so. Thus, in attempting to define analogous bifunctors for a $0$-abelian category $\C$ with enough injectives or projectives, we could ask what the analogues of such resolutions would be for objects in $\C$. Well, in the abelian case, these resolutions are, in particular, exact sequences, which can be viewed as splices of short exact sequences. Therefore, in the $0$-abelian case, it is reasonable to think that the analogues of such resolutions should consist of splices of bimorphisms. However, the splice of two bimorphisms is again a bimorphism.\footnote{This indeed makes sense, as $0 + 0 = 0$.} Hence it is expected that the analogues of injective and projective resolutions of an object $X \in \C$ are given by, respectively, bimorphisms $X \to I$ with $I$ injective and bimorphisms $P \to X$ with $P$ projective. Fortunately, under the appropriate assumptions, such bimorphisms do exist, as we show next.

\begin{proposition}\label{proposition.15}
Let $\B$ be an arbitrary category with enough injectives. If $\B$ has $0$-cokernels, then for every $X \in \B$ there is a bimorphism $X \to I$ in $\B$ with $I$ injective.
\end{proposition}

\begin{proof}
Consider $X \in \B$ and let $X \to J$ be a monomorphism in $\B$ with $J$ injective. Then its $0$-cokernel is a bimorphism $X \to I$ in $\B$ with $I$ injective.
\end{proof}

By duality, we deduce from Proposition \ref{proposition.15} that if $\B$ is a category with enough projectives and $0$-kernels, then for every $X \in \B$ there is a bimorphism $P \to X$ in $\B$ with $P$ projective. Moreover, it is easy to verify that, upon their existence, the bimorphisms $X \to I$ and $P \to X$ with $I$ injective and $P$ projective are uniquely determined by $X$ up to a unique isomorphism. Furthermore, note that, assuming the existence of these bimorphisms, the assignments of $X$ to $I$ and of $X$ to $P$ are functorial in $X$. Well, let us elaborate on these comments.

Recall that a subcategory $\A$ of a category $\B$ is called \textit{reflective} in $\B$ if the inclusion functor $\A \to \B$ has a left adjoint. Equivalently, $\A$ is reflective in $\B$ if and only if for every object $X \in \B$ there is a morphism $\eta_{X} \in \B(X,X_{\A})$ in $\B$ with $X_{\A} \in \A$ such that for every morphism $f \in \B(X,A)$ in $\B$ with $A \in \A$ there is a unique morphism $g \in \A(X_{\A},A)$ such that $f = g \eta_{X}$. If this is the case, then the assignment of $X$ to $X_{\A}$ defines a functor $(-)_{\A} : \B \to \B$ in such a way that the morphisms $\eta_{X}$ form a natural transformation $\eta : 1_{\B} \to (-)_{\A}$, which is the unit of the adjunction given by the left adjoint of the inclusion $\A \to \B$. Also, let us say that $\A$ is \textit{epimonoreflective} in $\B$ if $\A$ is reflective in $\B$ and the morphisms $\eta_{X}$ above (the components of the unit of the adjunction) are bimorphisms for every $X \in \B$.\footnote{Maybe the term ``bireflective'' would be more appropriate, but it seems that this is already used in the literature to refer to a subcategory that is both reflective and coreflective.} Dually, $\A$ is called \textit{coreflective} in $\B$ if the inclusion functor $\A \to \B$ has a right adjoint, and similar (dual) remarks to the ones above hold (which involve the counit of the adjunction). Finally, we say that $\A$ is \textit{epimonocoreflective} in $\B$ if $\A$ is coreflective in $\B$ and the components of the counit of the adjunction given by the right adjoint of the inclusion $\A \to \B$ are bimorphisms. For more details, see \cite{MR1712872}.

Given a category $\B$, let $\Injectives (\B)$ and $\Projectives (\B)$ be the subcategories of $\B$ consisting of its injective and projective objects, respectively. The next result ties up the previous two paragraphs.

\begin{proposition}\label{proposition.16}
Let $\B$ be an arbitrary category. The following are equivalent:
\begin{enumerate}
    \item[(a)] For every $X \in \B$ there is a bimorphism $X \to I$ in $\B$ with $I \in \Injectives (\B)$.
    \item[(b)] $\Injectives (\B)$ is epimonoreflective in $\B$.
    \item[(c)] $\Injectives (\B)$ is reflective in $\B$ and $\B$ has enough injectives.
\end{enumerate}
\end{proposition}

\begin{proof}
It is straightforward to conclude from the above discussion that (a) implies (b). Moreover, (b) trivially implies (c), and we prove below that (c) implies (a).

Suppose that $\Injectives (\B)$ is reflective in $\B$ and that $\B$ has enough injectives. Then for each $X \in \B$ there is a morphism $\eta_{X} \in \B(X,X_{\Injectives (\B)})$ in $\B$ with $X_{\Injectives (\B)} \in \Injectives (\B)$ such that for every morphism $f \in \B(X,I)$ in $\B$ with $I \in \Injectives (\B)$ there is a unique morphism $g \in \Injectives (\B) (X_{\Injectives (\B)}, I)$ such that $f = g \eta_{X}$. We claim that each $\eta_{X}$ is a bimorphism in $\B$. In fact, because $\B$ has enough injectives, there is a monomorphism $f \in \B(X,I)$ in $\B$ with $I \in \Injectives (\B)$, and since $f = g \eta_{X}$ for some morphism $g$ in $\B$, it follows that $\eta_{X}$ is also a monomorphism. Furthermore, if $v, w \in \B(X_{\Injectives (\B)}, V)$ are morphisms in $\B$ for which $v \eta_{X} = w \eta_{X}$, then $\eta_{V} v \eta_{X} = \eta_{V} w \eta_{X}$, and the universal property of $\eta_{X}$ implies that $\eta_{V} v = \eta_{V} w$. But we have just seen that $\eta_{V}$ is a monomorphism, hence $v = w$, so that $\eta_{X}$ is an epimorphism.
\end{proof}


By taking opposite categories in Proposition \ref{proposition.16}, we also deduce that, for a category $\B$, the following are equivalent:
\begin{enumerate}
    \item[(a)] For every $X \in \B$ there is a bimorphism $P \to X$ in $\B$ with $P \in \Projectives (\B)$.
    \item[(b)] $\Projectives (\B)$ is epimonocoreflective in $\B$.
    \item[(c)] $\Projectives (\B)$ is coreflective in $\B$ and $\B$ has enough projectives.
\end{enumerate}

Now, let us consider the above results for a $0$-abelian category. For the rest of this section, let us denote $\Injectives (\C)$ and $\Projectives (\C)$ simply by $\Injectives$ and $\Projectives$, respectively.

It follows from Theorem \ref{theorem.3} and Propositions \ref{proposition.15} and \ref{proposition.16} that if $\C$ is $0$-abelian and has enough injectives, then there is a functor $(-)_{\Injectives} : \C \to \C$ that takes values in $\Injectives$ and a natural transformation $\eta : 1_{\C} \to (-)_{\Injectives}$ such that $\eta_{X}$ is a bimorphism in $\C$ for every $X \in \C$. Similarly, if $\C$ is $0$-abelian and has enough projectives, then there is a functor $(-)_{\Projectives} : \C \to \C$ that takes values in $\Projectives$ and a natural transformation $\varepsilon : (-)_{\Projectives} \to 1_{\C}$ such that $\varepsilon_{X}$ is a bimorphism in $\C$ for every $X \in \C$. Also, note that, when they exist, the functors $(-)_{\Injectives}$ and $(-)_{\Projectives}$ are additive. With these remarks, we can now define the bifunctor $\E_{\C}(-,-)$ that we have mentioned in the beginning of this subsection.

Assume that $\C$ is a $0$-abelian category with enough injectives. Then, for each object $X \in \C$, choose a cokernel \[ \begin{tikzcd}
0 \arrow[r] & {\C(-,X)} \arrow[r, "{\C(-,\eta_{X})}"] &[1.8em] {\C(-,X_{\Injectives})} \arrow[r, "{\psi_{-,X}}"] &[0.8em] {\E_{\C}(-,X)} \arrow[r] & 0
\end{tikzcd} \] of $\C(-,\eta_{X})$ in $\mod \C$.\footnote{The author was very hesitant to introduce the label $\psi_{-,X}$ for the above cokernel, as it makes the notation heavier, but eventually realized that the clarity it provides justifies the loss of elegance.} Once these choices are made, observe that, given a morphism $f \in \C(X,Y)$ in $\C$, there is a unique morphism $\E_{\C}(-,f) \in \Hom(\E_{\C}(-,X), \E_{\C}(-,Y))$ making the diagram \[ \begin{tikzcd}
0 \arrow[r] & {\C(-,X)} \arrow[r, "{\C(-,\eta_{X})}"] \arrow[d, "{\C(-,f)}"'] &[1.8em] {\C(-,X_{\Injectives})} \arrow[r, "{\psi_{-,X}}"] \arrow[d, "{\C(-,f_{\Injectives})}"] &[0.8em] {\E_{\C}(-,X)} \arrow[r] \arrow[d, "{\E_{\C}(-,f)}"] & 0 \\
0 \arrow[r] & {\C(-,Y)} \arrow[r, "{\C(-,\eta_{Y})}"']                        & {\C(-,Y_{\Injectives})} \arrow[r, "{\psi_{-,Y}}"']                                      & {\E_{\C}(-,Y)} \arrow[r]                             & 0
\end{tikzcd} \] in $\mod \C$ commute. Clearly, the above assignments of an object $X \in \C$ to $\E_{\C}(-,X)$ and of a morphism $f$ in $\C$ to $\E_{\C}(-,f)$ define an additive functor $\C \to \mod \C$, which, in turn, defines a biadditive bifunctor $\E_{\C}(-,-) : \C^{\op} \times \C \to \Ab$.

It is convenient to point out that, under the conditions of the above paragraph, given $X,Y \in \C$, the abelian group $\E_{\C}(X,Y)$ is computed as follows: First, take a bimorphism $Y \to Y_{\Injectives}$ in $\C$ with $Y_{\Injectives}$ injective (which is unique up to a unique isomorphism). Then apply $\C(X,-)$ to this bimorphism. The cokernel of the resulting morphism in $\Ab$ is (up to isomorphism) the abelian group $\E_{\C}(X,Y)$. Equivalently, $\E_{\C}(X,Y)$ is the homology of the complex \[ \begin{tikzcd}
0 \arrow[r] & {\C(X,Y)} \arrow[r] & {\C(X,Y_{\Injectives})} \arrow[r] & 0
\end{tikzcd} \] at the position of $\C(X,Y_{\Injectives})$. Hopefully, the reader will agree that this procedure to compute $\E_{\C}(X,Y)$ resembles the computation of $\Ext_{\A}^{1}(A,B)$ via an injective resolution of $B$, where $\A$ is an abelian category with enough injectives and $A, B \in \A$.

When $\C$ is a $0$-abelian category with enough projectives, a similar reasoning can be carried out to define a biadditive bifunctor $\E'_{\C}(-,-) : \C^{\op} \times \C \to \Ab$. For the sake of clarity, let us write down the details, which, albeit repetitive, are instructive.

Suppose that $\C$ is $0$-abelian and has enough projectives. For each object $X  \in \C$, choose a cokernel \[ \begin{tikzcd}
0 \arrow[r] & {\C(X,-)} \arrow[r, "{\C(\varepsilon_{X},-)}"] &[1.8em] {\C(X_{\Projectives},-)} \arrow[r, "{\omega_{X,-}}"] &[0.8em] {\E'_{\C}(X,-)} \arrow[r] & 0
\end{tikzcd} \] of $\C(\varepsilon_{X},-)$ in $\mod \C^{\op}$. Then, for each morphism $f \in \C(X,Y)$ in $\C$, there is a unique morphism $\E'_{\C}(f,-) \in \Hom(\E'_{\C}(Y,-), \E'_{\C}(X,-))$ making the diagram \[ \begin{tikzcd}
0 \arrow[r] & {\C(Y,-)} \arrow[r, "{\C(\varepsilon_{Y},-)}"] \arrow[d, "{\C(f,-)}"'] &[1.8em] {\C(Y_{\Projectives},-)} \arrow[r, "{\omega_{Y,-}}"] \arrow[d, "{\C(f_{\Projectives},-)}"] &[0.8em] {\E'_{\C}(Y,-)} \arrow[r] \arrow[d, "{\E'_{\C}(f,-)}"] & 0 \\
0 \arrow[r] & {\C(X,-)} \arrow[r, "{\C(\varepsilon_{X},-)}"']                        & {\C(X_{\Projectives},-)} \arrow[r, "{\omega_{X,-}}"']                                       & {\E'_{\C}(X,-)} \arrow[r]                              & 0
\end{tikzcd} \] in $\mod \C$ commute. Thus, the assignments of an object $X \in \C$ to $\E'_{\C}(X,-)$ and of a morphism $f$ in $\C$ to $\E'_{\C}(f,-)$ define an additive functor $\C^{\op} \to \mod \C^{\op}$, which then defines a biadditive bifunctor $\E'_{\C}(-,-) : \C^{\op} \times \C \to \Ab$.

Under the assumptions of the above paragraph, note that, for $X,Y \in \C$, the abelian group $\E'_{\C}(X,Y)$ is computed as follows: Take a bimorphism $X_{\Projectives} \to X$ in $\C$ with $X_{\Projectives}$ projective (which is unique up to a unique isomorphism), then apply $\C(-,Y)$ to it. The cokernel of the resulting morphism in $\Ab$ is (up to isomorphism) the abelian group $\E'_{\C}(X,Y)$, which coincides with the homology of the complex \[ \begin{tikzcd}
0 \arrow[r] & {\C(X,Y)} \arrow[r] & {\C(X_{\Projectives},Y)} \arrow[r] & 0
\end{tikzcd} \] at the position of $\C(X_{\Projectives},Y)$. Finally, note that these steps to compute the abelian group $\E'_{\C}(X,Y)$ resemble the computation of $\Ext^{1}_{\A}(A,B)$ via a projective resolution of $A$, where $\A$ is an abelian category with enough projectives and $A, B \in \A$.

At this point, the reader might be asking whether the bifunctors $\E_{\C}(-,-)$ and $\E'_{\C}(-,-)$ coincide when $\C$ is a $0$-abelian category with enough injectives and enough projectives. Well, it turns out that they do coincide, as we show in Proposition \ref{proposition.18}. To prove this, we will make use of the following result:

\begin{proposition}\label{proposition.17}
Assume that $\C$ is a $0$-abelian category with enough injectives and enough projectives. Then $(-)_{\Injectives}$ is right adjoint to $(-)_{\Projectives}$.
\end{proposition}

\begin{proof}
Let $X,Y \in \C$ be arbitrary. Since $\eta_{Y}$ is a bimorphism and $X_{\Projectives}$ is projective, it follows that $\C(X_{\Projectives},\eta_{Y})$ is an isomorphism. Similarly, as $\varepsilon_{X}$ is a bimorphism and $Y_{\Injectives}$ is injective, $\C(\varepsilon_{X},Y_{\Injectives})$ is an isomorphism. Therefore, the composition \[ \C(\varepsilon_{X},Y_{\Injectives})^{-1} \C(X_{\Projectives},\eta_{Y}) : \C(X_{\Projectives},Y) \to \C(X,Y_{\Injectives}) \] is an isomorphism, which is easily checked to be natural in both $X$ and $Y$.
\end{proof}

\begin{proposition}\label{proposition.18}
Assume that $\C$ is a $0$-abelian category with enough injectives and enough projectives. Then the bifunctors $\E_{\C}(-,-)$ and $\E'_{\C}(-,-)$ are isomorphic.
\end{proposition}

\begin{proof}
Let $X,Y \in \C$ be arbitrary. By the proof of Proposition \ref{proposition.17}, there is an isomorphism $\C(X_{\Projectives},Y) \to \C(X,Y_{\Injectives})$, which we will denote by $\phi_{X,Y}$, that is natural in $X$ and $Y$, and makes the diagram \[ \begin{tikzcd}
{\C(X,Y)} \arrow[r, "{\C(\varepsilon_{X},Y)}"] \arrow[d, equal] &[1.8em] {\C(X_{\Projectives},Y)} \arrow[d, "{\phi_{X,Y}}"] \\
{\C(X,Y)} \arrow[r, "{\C(X,\eta_{Y})}"']                         & {\C(X,Y_{\Injectives})}                     
\end{tikzcd} \] in $\Ab$ commute. Therefore, there is a unique morphism $\delta_{X,Y} \in \Ab(\E'_{\C}(X,Y),\E_{\C}(X,Y))$ that makes the diagram below with exact rows in $\Ab$ commute. \[ \begin{tikzcd}
0 \arrow[r] & {\C(X,Y)} \arrow[r, "{\C(\varepsilon_{X},Y)}"] \arrow[d, equal] &[1.8em] {\C(X_{\Projectives},Y)} \arrow[r, "{\omega_{X,Y}}"] \arrow[d, "{\phi_{X,Y}}"] &[0.8em] {\E'_{\C}(X,Y)} \arrow[r] \arrow[d, "{\delta_{X,Y}}"] & 0 \\
0 \arrow[r] & {\C(X,Y)} \arrow[r, "{\C(X,\eta_{Y})}"']                        & {\C(X,Y_{\Injectives})} \arrow[r, "{\psi_{X,Y}}"']                             & {\E_{\C}(X,Y)} \arrow[r]                              & 0
\end{tikzcd} \] By the snake lemma or the circular sequence or the five lemma, $\delta_{X,Y}$ is an isomorphism. Moreover, observe that, from their definitions, the morphisms $\psi_{X,Y}$ and $\omega_{X,Y}$ are natural in $X$ and $Y$. Therefore, since we also know that $\omega_{X,Y}$ is an epimorphism and $\phi_{X,Y}$ is natural in $X$ and $Y$, we can conclude that so is $\delta_{X,Y}$.
\end{proof}

Due to Proposition \ref{proposition.18}, from now on, when $\C$ is a $0$-abelian category with enough projectives, we will denote $\E'_{\C}(-,-)$ by $\E_{\C}(-,-)$.\footnote{This is not entirely an abuse of notation, because if $\C$ is $0$-abelian and has enough injectives and enough projectives, then, using the notation in the proof of Proposition \ref{proposition.18}, when defining $\E'_{\C}(X,-)$ for $X \in \C$, we can choose the cokernel of $\C(\varepsilon_{X},-)$ to be the composition $\psi_{X,-} \phi_{X,-}$. With these choices, we get an equality $\E'_{\C}(-,-) = \E_{\C}(-,-)$.}

The ``balancing'' property of the bifunctor $\E_{\C}(-,-)$ given in Proposition \ref{proposition.18} (for a $0$-abelian category $\C$ with enough injectives and enough projectives) gives evidence for our claim that $\E_{\C}(-,-)$ is analogous to $\Ext^{1}_{\A}(-,-)$ for an abelian category $\A$. Another result that supports this claim is the following:

\begin{proposition}\label{proposition.19}
Let $\C$ be a $0$-abelian category with enough injectives. An object $X \in \C$ is injective if and only if $\E_{\C}(-,X) = 0$.
\end{proposition}

\begin{proof}
In fact, $X$ is injective if and only if $\eta_{X}$ is an isomorphism if and only if $\C(-,\eta_{X})$ is an isomorphism if and only if the cokernel of $\C(-,\eta_{X})$ in $\mod \C$ is zero.
\end{proof}

By duality, Proposition \ref{proposition.19} also shows that if $\C$ is a $0$-abelian category with enough projectives, then an object $X \in \C$ is projective if and only if $\E_{\C}(X,-) = 0$.

Next, we show in Theorem \ref{theorem.4} that every bimorphism in a $0$-abelian category with enough injectives or projectives induces a ``not so long exact sequence''. As the name suggests, this sequence is analogous to the classical long exact sequence induced by a short exact sequence in an abelian category $\A$, which involves the bifunctors $\Ext^{i}_{\A}(-,-)$. To prove Theorem \ref{theorem.4}, we will use the following lemma:

\begin{lemma}\label{lemma.1}
Let $\C$ be a $0$-abelian category with enough injectives. For a morphism $f$ in $\C$, the following hold:
\begin{enumerate}
    \item[(a)] If $f$ is a monomorphism, then $f_{\Injectives}$ is a split monomorphism and $\C(-,f)$ is a monomorphism in $\mod \C$.
    \item[(b)] If $f$ is an epimorphism, then $f_{\Injectives}$ is a split epimorphism and $\E_{\C}(-,f)$ is an epimorphism in $\mod \C$.
\end{enumerate}
\end{lemma}

\begin{proof}
Let $f \in \C(X,Y)$ be a morphism in $\C$, and consider the commutative diagram with exact rows and exact columns \[ \begin{tikzcd}
            & 0 \arrow[d]                                                     &[1.8em] 0 \arrow[d]                                                                            &[0.8em] 0 \arrow[d]                                          &   \\
0 \arrow[r] & K \arrow[d] \arrow[r]                                           & L \arrow[r] \arrow[d]                                                                  & M \arrow[d]                                          &   \\
0 \arrow[r] & {\C(-,X)} \arrow[r, "{\C(-,\eta_{X})}"] \arrow[d, "{\C(-,f)}"'] & {\C(-,X_{\Injectives})} \arrow[r, "{\psi_{-,X}}"] \arrow[d, "{\C(-,f_{\Injectives})}"] & {\E_{\C}(-,X)} \arrow[r] \arrow[d, "{\E_{\C}(-,f)}"] & 0 \\
0 \arrow[r] & {\C(-,Y)} \arrow[r, "{\C(-,\eta_{Y})}"'] \arrow[d]              & {\C(-,Y_{\Injectives})} \arrow[r, "{\psi_{-,Y}}"'] \arrow[d]                           & {\E_{\C}(-,Y)} \arrow[r] \arrow[d]                   & 0 \\
            & F \arrow[d] \arrow[r]                                           & G \arrow[d] \arrow[r]                                                                  & H \arrow[d] \arrow[r]                                & 0 \\
            & 0                                                               & 0                                                                                      & 0                                                    &  
\end{tikzcd} \] in $\mod \C$, which is obtained by taking the kernels and cokernels of $\C(-,f)$, $\C(-,f_{\Injectives})$ and $\E_{\C}(-,f)$. By Proposition \ref{proposition.3}, $\C$ is right comprehensive, hence $\C(-,X_{\Injectives})$ and $\C(-,Y_{\Injectives})$ are both projective and injective in $\mod \C$. Therefore, as $\gldim (\mod \C) \leqslant 1$, we can conclude that $L$ and $G$ are also projective and injective in $\mod \C$. Hence there are $I, J \in \C$ such that $L \simeq \C(-,I)$ and $G \simeq \C(-,J)$.\footnote{Note that $I$ and $J$ are injective in $\C$, by \cite[Proposition 7.1]{2409.10438}. We will not need this fact, though.}

Now, suppose that $f$ is a monomorphism. Then $\C(-,f)$ is a monomorphism in $\mod \C$, so that $K = 0$, which implies that $L \to M$ is a monomorphism. Thus, the composition $L \to M \to \E_{\C}(-,X)$ is a monomorphism, which must split since $L$ is injective. Consequently, there is a split epimorphism $\E_{\C}(-,X) \to L$. However, because $\E_{\C}(-,X) \simeq \mr(\eta_{X})$ and $\eta_{X}$ is an epimorphism, it follows from Proposition \ref{proposition.6} that $\E_{\C}(-,X)^{\ast} = 0$. Therefore, as $L \simeq \C(-,I)$, we conclude that $\Hom(\E_{\C}(-,X),L) = 0$, so that the split epimorphism $\E_{\C}(-,X) \to L$ is the zero morphism, which implies that $L = 0$. Thus, $\C(-,f_{\Injectives})$ is a monomorphism in $\mod \C$, and because $\C(-,X_{\Injectives})$ is injective, it must split. Hence $f_{\Injectives}$ is a split monomorphism.

Next, assume that $f$ is an epimorphism. Since $\eta_{Y}$ is also an epimorphism and $f_{\Injectives} \eta_{X} = \eta_{Y} f$, we obtain that $f_{\Injectives}$ is an epimorphism. Therefore, as $G \simeq \mr(f_{\Injectives})$, it follows from Proposition \ref{proposition.6} that $G^{\ast} = 0$. But $G \simeq \C(-,J)$, hence $\Hom(G,G) = 0$, and we conclude that $G = 0$, which implies that $H = 0$. Consequently, $\C(-,f_{\Injectives})$ and $\E_{\C}(-,f)$ are both epimorphisms in $\mod \C$. Moreover, given that $\C(-,Y_{\Injectives})$ is projective, $\C(-,f_{\Injectives})$ is a split epimorphism, hence so is $f_{\Injectives}$.
\end{proof}

Note that, by taking opposite categories in Lemma \ref{lemma.1}, we can deduce that if $\C$ is $0$-abelian and has enough projectives, then, for a morphism $f$ in $\C$, the following hold:
\begin{enumerate}
    \item[(a)] If $f$ is an epimorphism, then $f_{\Projectives}$ is a split epimorphism and $\C(f,-)$ is a monomorphism in $\mod \C^{\op}$.
    \item[(b)] If $f$ is a monomorphism, then $f_{\Projectives}$ is a split monomorphism and $\E_{\C}(f,-)$ is an epimorphism in $\mod \C^{\op}$.
\end{enumerate}

\begin{theorem}[Not so long exact sequence]\label{theorem.4}
Let $\C$ be a $0$-abelian category with enough injectives. Every bimorphism $f \in \C(X,Y)$ in $\C$ induces an exact sequence \[ \begin{tikzcd}
0 \arrow[r] & {\C(-,X)} \arrow[r, "{\C(-,f)}"]           &[1.6em] {\C(-,Y)} \arrow[out=0, in=180, looseness=2, overlay, ld]     &   \\
            & {\E_{\C}(-,X)} \arrow[r, "{\E_{\C}(-,f)}"] & {\E_{\C}(-,Y)} \arrow[r] & 0
\end{tikzcd} \] in $\mod \C$.
\end{theorem}

\begin{proof}
Let $f \in \C(X,Y)$ be a bimorphism in $\C$. It follows from Lemma \ref{lemma.1} that $f_{\Injectives}$ is an isomorphism, $\C(-,f)$ is a monomorphism in $\mod \C$ and $\E_{\C}(-,f)$ is an epimorphism in $\mod \C$. Therefore, given the commutative diagram with exact rows \[ \begin{tikzcd}
0 \arrow[r] & {\C(-,X)} \arrow[r, "{\C(-,\eta_{X})}"] \arrow[d, "{\C(-,f)}"'] &[1.8em] {\C(-,X_{\Injectives})} \arrow[r, "{\psi_{-,X}}"] \arrow[d, "{\C(-,f_{\Injectives})}"] &[0.8em] {\E_{\C}(-,X)} \arrow[r] \arrow[d, "{\E_{\C}(-,f)}"] & 0 \\
0 \arrow[r] & {\C(-,Y)} \arrow[r, "{\C(-,\eta_{Y})}"']                        & {\C(-,Y_{\Injectives})} \arrow[r, "{\psi_{-,Y}}"']                                      & {\E_{\C}(-,Y)} \arrow[r]                             & 0
\end{tikzcd} \] in $\mod \C$, if we set $\partial(-,f) = \psi_{-,X} \C(-,f_{\Injectives})^{-1} \C(-,\eta_{Y})$, then it is easy to check that $\C(-,f)$ is a kernel of $\partial(-,f)$ and that $\E_{\C}(-,f)$ is a cokernel of $\partial(-,f)$. Consequently, \[ \begin{tikzcd}
0 \arrow[r] & {\C(-,X)} \arrow[r, "{\C(-,f)}"] &[1.2em] {\C(-,Y)} \arrow[r, "{\partial{(-,f)}}"] &[1.2em] {\E_{\C}(-,X)} \arrow[r, "{\E_{\C}(-,f)}"] &[1.6em] {\E_{\C}(-,Y)} \arrow[r] & 0
\end{tikzcd} \] is an exact sequence in $\mod \C$.
\end{proof}

By replacing $\C$ by $\C^{\op}$ in Theorem \ref{theorem.4}, we conclude that if $\C$ is $0$-abelian and has enough projectives, then every bimorphism $f \in \C(X,Y)$ in $\C$ induces an exact sequence \[ \begin{tikzcd}
0 \arrow[r] & {\C(Y,-)} \arrow[r, "{\C(f,-)}"]           &[1.6em] {\C(X,-)} \arrow[out=0, in=180, looseness=2, overlay, ld]     &   \\
            & {\E_{\C}(Y,-)} \arrow[r, "{\E_{\C}(f,-)}"] & {\E_{\C}(X,-)} \arrow[r] & 0
\end{tikzcd} \] in $\mod \C^{\op}$. We will use the name \textit{not so long exact sequence} to refer to both the above sequence in $\mod \C^{\op}$ and to the sequence in $\mod \C$ from Theorem \ref{theorem.4}, and also to any of their evaluations at an object in $\C$, which gives an exact sequence in $\Ab$.

There is an obvious feature of the not so long exact sequence that needs to be mentioned, namely, that it terminates immediately, as opposed to the classical long exact sequence induced by a short exact sequence in an abelian category, which can be infinite. This fact suggests that, while $\E_{\C}(-,-)$ is analogous to $\Ext^{1}_{\A}(-,-)$ (for $\C$ a $0$-abelian category with enough injectives or projectives and $\A$ an abelian category), there are no analogues of $\Ext^{i}_{\A}(-,-)$ for $\C$ when $i \geqslant 2$. Or, if there are, they are just zero. In fact, Proposition \ref{proposition.22} supports this claim, since when $\A$ has enough injectives and enough projectives, there is an isomorphism of abelian groups $\Ext^{i}_{\A}(-,B) \otimes \Ext^{j}_{\A}(A,-) \simeq \Ext^{i+j}_{\A}(A,B)$ for every $i,j \geqslant 1$ and $A,B \in \A$.\footnote{Note that the conditions that $\A$ has enough injectives and enough projectives guarantee that $\Ext^{i}_{\A}(-,B) \in \mod \A$ and $\Ext^{j}_{\A}(A,-) \in \mod \A^{\op}$, respectively. Indeed, one can use an injective resolution of $B$ in $\A$ and a projective resolution of $A$ in $\A$ (and shifting arguments) to obtain a projective presentation of $\Ext^{i}_{\A}(-,B)$ and a projective presentation of $\Ext^{j}_{\A}(A,-)$, respectively. Thus, one can prove the isomorphism $\Ext^{i}_{\A}(-,B) \otimes \Ext^{j}_{\A}(A,-) \simeq \Ext^{i+j}_{\A}(A,B)$ by proceeding similarly to the proof of Proposition \ref{proposition.22}.}

\begin{proposition}\label{proposition.22}
Let $\C$ be a $0$-abelian category with enough injectives and enough projectives. Then $\E_{\C}(-,Y) \otimes \E_{\C}(X,-) = 0$ for every $X,Y \in \C$.
\end{proposition}

\begin{proof}
Let $X,Y \in \C$ be arbitrary objects, and consider the short exact sequence \[ \begin{tikzcd}
0 \arrow[r] & {\C(-,Y)} \arrow[r, "{\C(-,\eta_{Y})}"] &[1.8em] {\C(-,Y_{\Injectives})} \arrow[r, "{\psi_{-,Y}}"] &[0.8em] {\E_{\C}(-,Y)} \arrow[r] & 0
\end{tikzcd} \] in $\mod \C$. Given that the functor $- \otimes \E_{\C}(X,-) : \mod \C \to \Ab$ preserves cokernels and satisfies that $\C(-,Z) \otimes \E_{\C}(X,-) \simeq \E_{\C}(X,Z)$ for every $Z \in \C$, see \cite[Appendix A]{2409.10438}, if we apply it to the above sequence, we obtain an exact sequence \[ \begin{tikzcd}
{\E_{\C}(X,Y)} \arrow[r] & {\E_{\C}(X,Y_{\Injectives})} \arrow[r] & {\E_{\C}(-,Y) \otimes \E_{\C}(X,-)} \arrow[r] & 0
\end{tikzcd} \] in $\Ab$. However, because $Y_{\Injectives}$ is injective, we have $\E_{\C}(X,Y_{\Injectives}) = 0$, by Proposition \ref{proposition.19}. Consequently, $\E_{\C}(-,Y) \otimes \E_{\C}(X,-) = 0$.
\end{proof}

\subsection{Their injective finitely presented modules}

Motivated by Theorem \ref{theorem.2}, we show next how the construction of the bifunctor $\E_{\C}(-,-)$ and the not so long exact sequence give injective resolutions of objects in $\mod \C$ and in $\mod \C^{\op}$, whenever $\C$ is a $0$-abelian category with enough injectives and enough projectives, respectively. Furthermore, by assuming these respective conditions on $\C$, we classify the injective objects in $\mod \C$ and $\mod \C^{\op}$ in Corollary \ref{corollary.3}, after proving a $0$-abelian version of a conjecture due to Auslander in Theorem \ref{theorem.5}.

\begin{lemma}\label{lemma.2}
Let $\C$ be a $0$-abelian category with enough injectives. Then $\E_{\C}(-,X)$ is injective in $\mod \C$ for every $X \in \C$.
\end{lemma}

\begin{proof}
Given $X \in \C$, consider the exact sequence \[ \begin{tikzcd}
0 \arrow[r] & {\C(-,X)} \arrow[r, "{\C(-,\eta_{X})}"] &[1.8em] {\C(-,X_{\Injectives})} \arrow[r, "{\psi_{-,X}}"] &[0.8em] {\E_{\C}(-,X)} \arrow[r] & 0
\end{tikzcd} \] in $\mod \C$. Since $X_{\Injectives}$ is injective and, by Proposition \ref{proposition.3}, $\C$ is right comprehensive, $\C(-,X_{\Injectives})$ is injective in $\mod \C$. Therefore, as $\gldim (\mod \C) \leqslant 1$, it follows that $\E_{\C}(-,X)$ is injective in $\mod \C$.
\end{proof}

By duality, Lemma \ref{lemma.2} also says that if $\C$ is $0$-abelian and has enough projectives, then $\E_{\C}(X,-)$ is injective in $\mod \C^{\op}$ for every $X \in \C$.

\begin{proposition}\label{proposition.21}
Let $\C$ be a $0$-abelian category with enough injectives. Every $F \in \mod \C$ admits an injective resolution in $\mod \C$ of the form \[ \begin{tikzcd}[ampersand replacement=\&]
0 \arrow[r] \& F \arrow[r] \& {\E_{\C}(-,X) \oplus \C(-,Z_{\Injectives})} \arrow[r, "{\left(\begin{smallmatrix}
        \E_{\C}(-,f) & 0 \\ 0 & \psi_{-,Z}
    \end{smallmatrix}\right)}"] \&[4.8em] {\E_{\C}(-,Y) \oplus \E_{\C}(-,Z)} \arrow[r] \& 0
\end{tikzcd} \] where $f \in \C(X,Y)$ is a bimorphism in $\C$ and $Z \in \C$ satisfies $F^{\ast} \simeq \C(Z,-)$.
\end{proposition}

\begin{proof}
Let $F \in \mod \C$ be arbitrary. It follows from Proposition \ref{proposition.13} that there is a bimorphism $f \in \C(X,Y)$ in $\C$ and an object $Z \in \C$ such that $F \simeq \mr(f) \oplus \C(-,Z)$. In this case, observe that, as it was explained in the proof of Lemma \ref{lemma.2}, \[ \begin{tikzcd}
0 \arrow[r] & {\C(-,Z)} \arrow[r, "{\C(-,\eta_{Z})}"] &[1.8em] {\C(-,Z_{\Injectives})} \arrow[r, "{\psi_{-,Z}}"] &[0.8em] {\E_{\C}(-,Z)} \arrow[r] & 0
\end{tikzcd} \] is an injective resolution of $\C(-,Z)$ in $\mod \C$. Moreover, by Theorem \ref{theorem.4}, there is an exact sequence \[ \begin{tikzcd}
0 \arrow[r] & \mr(f) \arrow[r] & {\E_{\C}(-,X)} \arrow[r, "{\E_{\C}(-,f)}"] &[1.6em] {\E_{\C}(-,Y)} \arrow[r] & 0
\end{tikzcd} \] in $\mod \C$, which is an injective resolution of $\mr(f)$, by Lemma \ref{lemma.2}. Therefore, by taking the direct sum of the above sequences, we obtain an injective resolution \[ \begin{tikzcd}[ampersand replacement=\&]
0 \arrow[r] \& F \arrow[r] \& {\E_{\C}(-,X) \oplus \C(-,Z_{\Injectives})} \arrow[r, "{\left(\begin{smallmatrix}
        \E_{\C}(-,f) & 0 \\ 0 & \psi_{-,Z}
    \end{smallmatrix}\right)}"] \&[4.8em] {\E_{\C}(-,Y) \oplus \E_{\C}(-,Z)} \arrow[r] \& 0
\end{tikzcd} \] of $F$ in $\mod \C$. Finally, since $f$ is an epimorphism, Proposition \ref{proposition.6} gives that $\mr(f)^{\ast} = 0$. Thus, as $\C(-,Z)^{\ast} \simeq \C(Z,-)$, we obtain that $F^{\ast} \simeq \C(Z,-)$.
\end{proof}

For the sake of completeness, let us mention the dual result of Proposition \ref{proposition.21}. Namely, if $\C$ is $0$-abelian and has enough projectives, then every $F \in \mod \C^{\op}$ admits an injective resolution \[ \begin{tikzcd}[ampersand replacement=\&]
0 \arrow[r] \& F \arrow[r] \& {\E_{\C}(Y,-) \oplus \C(Z_{\Projectives},-)} \arrow[r, "{\left(\begin{smallmatrix}
        \E_{\C}(f,-) & 0 \\ 0 & \omega_{Z,-}
    \end{smallmatrix}\right)}"] \&[4.8em] {\E_{\C}(X,-) \oplus \E_{\C}(Z,-)} \arrow[r] \& 0
\end{tikzcd} \] in $\mod \C^{\op}$, where $f \in \C(X,Y)$ is a bimorphism in $\C$ and $Z \in \C$ satisfies $F^{\ast} \simeq \C(-,Z)$.

It follows from the proof of Proposition \ref{proposition.21} that if $\C$ is $0$-abelian and has enough injectives, then every injective object in $\mod \C$ is of the form $\mr(f) \oplus \C(-,I)$ for some bimorphism $f \in \C(X,Y)$ in $\C$ such that $\E_{\C}(-,f)$ is a split epimorphism and for some injective object $I \in \C$. In particular, $\mr(f)$ is direct summand of $\E_{\C}(-,X)$. Thus, with the goal of classifying the injective objects in $\mod \C$, it becomes natural to ask whether a direct summand of $\E_{\C}(-,X)$ is also of the form $\E_{\C}(-,V)$ for some $V \in \C$.

The above question can be regarded as a $0$-abelian version of a conjecture attributed to Auslander, which originated in \cite[Section 4]{MR0212070}. It states that every direct summand of $\Ext_{\A}^{1}(-,A)$ is isomorphic to $\Ext_{\A}^{1}(-,B)$ for some $B \in \A$, whenever $\A$ is an abelian category with enough injectives and $A \in \A$ is arbitrary.\footnote{By considering opposite categories, we conclude that the dual version of this conjecture is that every direct summand of $\Ext_{\A}^{1}(A,-)$ is isomorphic to $\Ext_{\A}^{1}(B,-)$ for some $B \in \A$, whenever $\A$ is an abelian category with enough projectives and $A \in \A$ is arbitrary.} For the amusement of the reader, let us give a brief overview of this problem, see also \cite[Section 1]{MR3516080}. In the same paper, in \cite[Proposition 4.7]{MR0212070}, Auslander proved that the conjecture holds if $\A$ has finite global dimension, while Freyd proved the conjecture in \cite{MR206069} for the case when $\A$ \textit{has countable powers}, that is, when for every $A \in \A$, the product of countably many copies of $A$ exists in $\A$. However, shortly after these results were given, Auslander showed in \cite{MR237606} that the conjecture is not true in general. Still, it remains an open problem to characterize the categories $\A$ for which the conjecture holds. A more recent contribution to this problem was given by Martsinkovsky in \cite{MR3516080}, where he proved in \cite[Proposition 17]{MR3516080} that the conjecture is true when $\A$ is a \textit{length category}, that is, when the collection of subobjects of each object in $\A$ satisfies both the descending and ascending chain conditions.\footnote{Actually, the conjecture was proved for a larger class of categories, see \cite[Remark 4]{MR3516080}.}

For the $0$-abelian version of the above conjecture, the story is much shorter. Indeed, as we show next, it is just always true.

\begin{theorem}\label{theorem.5}
Let $\C$ be a $0$-abelian category with enough injectives. Given $X \in \C$, every direct summand of $\E_{\C}(-,X)$ in $\mod \C$ is isomorphic to $\E_{\C}(-,V)$ for some $V \in \C$. Furthermore, $V$ can be chosen in such a way that there is a bimorphism $h \in \C(X,V)$ for which $\E_{\C}(-,h)$ is a split epimorphism.\footnote{Our proof also shows that, under the assumptions of the theorem, if $F \in \eff \C$ and if there is an epimorphism $\E_{\C}(-,X) \to F$ in $\mod \C$, then $F \simeq \E_{\C}(-,V)$ for some $V \in \C$ such that $X_{\Injectives} \simeq V_{\Injectives}$.}
\end{theorem}

\begin{proof}
Let $F \in \mod \C$ be a direct summand of $\E_{\C}(-,X)$, so that there is a split epimorphism $\alpha \in \Hom(\E_{\C}(-,X),F)$. Since there is an exact sequence \[ \begin{tikzcd}
0 \arrow[r] & {\C(-,X)} \arrow[r, "{\C(-,\eta_{X})}"] &[1.8em] {\C(-,X_{\Injectives})} \arrow[r, "{\psi_{-,X}}"] &[0.8em] {\E_{\C}(-,X)} \arrow[r] & 0
\end{tikzcd} \] in $\mod \C$ and $\eta_{X}$ is an epimorphism, it follows from Proposition \ref{proposition.6} that $\E_{\C}(-,X)^{\ast} = 0$. Thus, given that $F \oplus \Ker \alpha \simeq \E_{\C}(-,X)$, we obtain that $F^{\ast} = 0$ and $(\Ker \alpha)^{\ast} = 0$. Next, note that, as $\gldim (\mod \C) \leqslant 1$, there is a commutative diagram with exact rows \[ \begin{tikzcd}
0 \arrow[r] & {\C(-,X)} \arrow[r, "{\C(-,\eta_{X})}"] \arrow[d, "{\C(-,h)}"'] &[1.8em] {\C(-,X_{\Injectives})} \arrow[r, "{\psi_{-,X}}"] \arrow[d, equal] &[0.8em] {\E_{\C}(-,X)} \arrow[r] \arrow[d, "\alpha"] & 0 \\
0 \arrow[r] & {\C(-,V)} \arrow[r, "{\C(-,g)}"']                               & {\C(-,X_{\Injectives})} \arrow[r, "{\alpha \psi_{-,X}}"']          & F \arrow[r]                                  & 0
\end{tikzcd} \] in $\mod \C$ for some morphisms $g \in \C(V,X_{\Injectives})$ and $h \in \C(X,V)$ in $\C$. In this case, $g$ is a monomorphism for which $F \simeq \mr(g)$ and, by the snake lemma or the circular sequence, $h$ is a monomorphism such that $\Ker \alpha \simeq \mr(h)$. Consequently, by Proposition \ref{proposition.6}, $g$ and $h$ are epimorphisms, and hence bimorphisms. In particular, it follows from Lemma \ref{lemma.1} that $g_{\Injectives}$ is an isomorphism. Hence, from $g_{\Injectives} \eta_{V} = g$, we get that $\C(-,\eta_{V}) = \C(-,g_{\Injectives}^{-1}) \C(-,g)$, so that there is a unique morphism $\beta \in \Hom(F, \E_{\C}(-,V))$ making the diagram below with exact rows in $\mod \C$ commute. \[ \begin{tikzcd}
0 \arrow[r] & {\C(-,V)} \arrow[r, "{\C(-,g)}"] \arrow[d, equal] &[1.8em] {\C(-,X_{\Injectives})} \arrow[r, "{\alpha \psi_{-,X}}"] \arrow[d, "{\C(-,g_{\Injectives}^{-1})}"] &[0.8em] F \arrow[r] \arrow[d, "\beta"] & 0 \\
0 \arrow[r] & {\C(-,V)} \arrow[r, "{\C(-,\eta_{V})}"']          & {\C(-,V_{\Injectives})} \arrow[r, "{\psi_{-,V}}"']                                                 & {\E_{\C}(-,V)} \arrow[r]       & 0
\end{tikzcd} \] But as $g_{\Injectives}^{-1}$ is an isomorphism, so is $\C(-,g_{\Injectives}^{-1})$, hence it follows from the snake lemma or the circular sequence or the five lemma that $\beta$ is an isomorphism.

Finally, observe that, by ``composing'' the two diagrams above, we obtain a commutative diagram with exact rows \[ \begin{tikzcd}
0 \arrow[r] & {\C(-,X)} \arrow[r, "{\C(-,\eta_{X})}"] \arrow[d, "{\C(-,h)}"'] &[1.8em] {\C(-,X_{\Injectives})} \arrow[r, "{\psi_{-,X}}"] \arrow[d, "{\C(-,g_{\Injectives}^{-1})}"] &[0.8em] {\E_{\C}(-,X)} \arrow[r] \arrow[d, "\beta \alpha"] & 0 \\
0 \arrow[r] & {\C(-,V)} \arrow[r, "{\C(-,\eta_{V})}"']                        & {\C(-,V_{\Injectives})} \arrow[r, "{\psi_{-,V}}"']                                          & {\E_{\C}(-,V)} \arrow[r]                           & 0
\end{tikzcd} \] in $\mod \C$. By the uniqueness of the morphisms involved in this diagram, we conclude that $g_{\Injectives}^{-1} = h_{\Injectives}$ and $\beta \alpha = \E_{\C}(-,h)$. Therefore, $\E_{\C}(-,h)$ is a split epimorphism.
\end{proof}

By duality, Theorem \ref{theorem.5} also states that if $\C$ is $0$-abelian and has enough projectives, then, for $X \in \C$, every direct summand of $\E_{\C}(X,-)$ in $\mod \C^{\op}$ is isomorphic to $\E_{\C}(V,-)$ for some $V \in \C$. Moreover, such $V$ can be chosen in a way that there is a bimorphism $h \in \C(V,X)$ for which $\E_{\C}(h,-)$ is a split epimorphism.

We can now reach a full classification of the injective objects in $\mod \C$ when $\C$ is a $0$-abelian category with enough injectives. Indeed, this is given by the next corollary, since every object in $\mod \C$ is the direct sum of a projective and an effaceable object in $\mod \C$ when $\C$ is $0$-abelian, see Proposition \ref{proposition.13}.

\begin{corollary}\label{corollary.3}
Let $\C$ be a $0$-abelian category with enough injectives.
\begin{enumerate}
    \item[(a)] The objects in $\mod \C$ that are injective and projective are (up to isomorphism) given by $\C(-,I)$, where $I \in \C$ is injective in $\C$.
    \item[(b)] The objects in $\mod \C$ that are injective and effaceable are (up to isomorphism) given by $\E_{\C}(-,X)$, where $X \in \C$ is arbitrary.
\end{enumerate}
\end{corollary}

\begin{proof}
It follows from \cite[Proposition 7.1]{2409.10438} that every object in $\mod \C$ that is injective and projective is isomorphic to $\C(-,I)$ for some injective object $I \in \C$. Moreover, all the objects in $\mod \C$ of this form are projective and injective since $\C$ is right comprehensive, by Proposition \ref{proposition.3}. Hence item (a) holds. Below, we prove item (b).

Suppose that $F \in \mod \C$ is injective and effaceable. Since $\gldim (\mod \C) \leqslant 1$, there is a monomorphism $f \in \C(X,Y)$ in $\C$ for which $F \simeq \mr(f)$. Then, by Proposition \ref{proposition.6}, $f$ is a bimorphism, and Theorem \ref{theorem.4} gives a short exact sequence \[ \begin{tikzcd}
0 \arrow[r] & F \arrow[r] & {\E_{\C}(-,X)} \arrow[r, "{\E_{\C}(-,f)}"] &[1.6em] {\E_{\C}(-,Y)} \arrow[r] & 0
\end{tikzcd} \] in $\mod \C$, which splits, as $F$ is injective. Thus, we conclude from Theorem \ref{theorem.5} that $F \simeq \E_{\C}(-,V)$ for some $V \in \C$. Also, it follows from their definition, Proposition \ref{proposition.6} and Lemma \ref{lemma.2} that $\E_{\C}(-,V)$ is effaceable and injective in $\mod \C$ for every $V \in \C$.
\end{proof}

By taking $\C^{\op}$ in place of $\C$ in Corollary \ref{corollary.3}, we conclude that if $\C$ is $0$-abelian and has enough projectives, then the projective injective objects in $\mod \C^{\op}$ are given by $\C(P,-)$, where $P \in \C$ is projective in $\C$, while the effaceable injective objects in $\mod \C^{\op}$ are given by $\E_{\C}(X,-)$ for any $X \in \C$. Due to the dual result of Proposition \ref{proposition.13}, these facts classify the injective objects in $\mod \C^{\op}$ when $\C$ is a $0$-abelian category with enough projectives.

\subsection{Their stable categories}

Following \cite[Subsection 2.3]{2409.10438}, let $\langle \Injectives \rangle$ and $\langle \Projectives \rangle$ be the ideals of $\C$ consisting of the morphisms in $\C$ that factor through $\Injectives$ and $\Projectives$, respectively. Recall that the \textit{injectively stable category} of $\C$ is the quotient $\C / \langle \Injectives \rangle$, while the \textit{projectively stable category} of $\C$ is given by $\C / \langle \Projectives \rangle$, which we denote simply by $\overline{\C}$ and $\underline{\C}$, respectively. In this subsection, we investigate under what conditions $\overline{\C}$ and $\underline{\C}$ are $n$-abelian categories, assuming that $\C$ is a $0$-abelian category with, respectively, enough injectives and enough projectives. Our main tool to study this problem will be, under the previous respective assumptions on $\C$, an embedding of $\overline{\C}$ in $\mod \C$ and of $\underline{\C}$ in $\mod \C^{\op}$. These embeddings are provided by a $0$-abelian version of the Hilton-Rees theorem, which we present in Theorem \ref{theorem.6}. Furthermore, we give a couple of results on the existence of $n$-kernels and $n$-cokernels on the stable categories of an abelian category with enough injectives or projectives, which are motivated by our discussion on $0$-abelian categories.

Let us start by giving a $0$-abelian version of the Hilton-Rees theorem, a classical result for abelian categories proved in \cite{MR124377}, see \cite[Section 4]{MR3516080}.

\begin{theorem}\label{theorem.6}
Let $\C$ be a $0$-abelian category with enough injectives. For each $X,Y \in \C$ there is an isomorphism of abelian groups \[ \overline{\C}(X,Y) \to \Hom(\E_{\C}(-,X), \E_{\C}(-,Y)) \] given by $\overline{f} \mapsto \E_{\C}(-,f)$, which is natural in both $X$ and $Y$. These isomorphisms define a fully faithful additive functor $\overline{\C} \to \mod \C$.
\end{theorem}

\begin{proof}
As we already pointed out when we defined $\E_{\C}(-,-)$ in Subsection \ref{subsection.4.2}, the assignments of an object $X \in \C$ to $\E_{\C}(-,X)$ and of a morphism $f$ in $\C$ to $\E_{\C}(-,f)$ define an additive functor $\C \to \mod \C$, which is easily seen to be full. In fact, let $X,Y \in \C$ be arbitrary, and take $\alpha \in \Hom(\E_{\C}(-,X), \E_{\C}(-,Y))$. Then there are morphisms $g \in \C(X_{\Injectives},Y_{\Injectives})$ and $f \in \C(X,Y)$ for which the diagram below with exact rows in $\mod \C$ commutes. \[ \begin{tikzcd}
0 \arrow[r] & {\C(-,X)} \arrow[r, "{\C(-,\eta_{X})}"] \arrow[d, "{\C(-,f)}"'] &[1.8em] {\C(-,X_{\Injectives})} \arrow[r, "{\psi_{-,X}}"] \arrow[d, "{\C(-,g)}"] &[0.8em] {\E_{\C}(-,X)} \arrow[r] \arrow[d, "\alpha"] & 0 \\
0 \arrow[r] & {\C(-,Y)} \arrow[r, "{\C(-,\eta_{Y})}"']                        & {\C(-,Y_{\Injectives})} \arrow[r, "{\psi_{-,Y}}"']                                          & {\E_{\C}(-,Y)} \arrow[r]                           & 0
\end{tikzcd} \] However, from the uniqueness of $f_{\Injectives}$, we get $g = f_{\Injectives}$, which implies that $\alpha = \E_{\C}(-,f)$. Now, since the kernel $\mathsf{K}$ of the functor $\C \to \mod \C$ induces a fully faithful additive functor $\C / \mathsf{K} \to \mod \C$, it suffices to prove that $\mathsf{K} = \langle \Injectives \rangle$ to conclude the proof of the theorem. So, let us show that $\mathsf{K} = \langle \Injectives \rangle$.

Fixed $X, Y \in \C$, we must verify that $\mathsf{K}(X,Y) = \langle \Injectives \rangle (X,Y)$, that is, that for a given $f \in \C(X,Y)$, we have $\E_{\C}(-,f) = 0$ if and only if $f$ factors through an injective object in $\C$. But note that, following the previous diagram, $\E_{\C}(-,f) = 0$ if and only if $\C(-,f_{\Injectives})$ factors through $\C(-,\eta_{Y})$, which happens if and only if $f_{\Injectives}$ factors through  $\eta_{Y}$. Thus, as $f_{\Injectives}$ factors through $\eta_{Y}$ if and only if $f$ factors through $\eta_{X}$, which occurs if and only if $f$ factors through some injective object in $\C$, we are done.
\end{proof}

By replacing $\C$ by $\C^{\op}$ in Theorem \ref{theorem.6}, we also conclude that if $\C$ is a $0$-abelian category with enough projectives, then there is a fully faithful functor $\underline{\C} \to \mod \C^{\op}$ which induces, for each $X, Y \in \C$, an isomorphism of abelian groups \[ \underline{\C}(X,Y) \to \Hom(\E_{\C}(Y,-),\E_{\C}(X,-)) \] given by $\underline{f} \mapsto \E_{\C}(f,-)$, which is natural in both $X$ and $Y$.

The embedding of $\overline{\C}$ into $\mod \C$ provided by Theorem \ref{theorem.6} when $\C$ is a $0$-abelian category with enough injectives turns out to be very useful. Indeed, it allows us to investigate properties of $\overline{\C}$ via the category $\mod \C$, which, by now, we have a better understanding of. For instance, we can use this embedding to prove the following:

\begin{proposition}\label{proposition.20}
Let $\C$ be a $0$-abelian category with enough injectives. Then $\overline{\C}$ is additive, idempotent complete, and has $0$-cokernels.
\end{proposition}

\begin{proof}
It is clear that $\overline{\C}$ is additive, and it follows from Theorems \ref{theorem.5} and \ref{theorem.6} that $\overline{\C}$ is idempotent complete. While we leave the details of these claims to the reader, we show below that $\overline{\C}$ has $0$-cokernels.\footnote{Nevertheless, observe that the existence of $0$-cokernels in $\overline{\C}$ also implies idempotent completeness for $\overline{\C}$, by Proposition \ref{proposition.2}.}

Let $\overline{f} \in \overline{\C}(X,Y)$ be an arbitrary morphism in $\overline{\C}$, where $f \in \C(X,Y)$ is one of its representatives. Then take a factorization $\E_{\C}(-,f) = \beta \alpha$ in $\mod \C$, where $\alpha \in \Hom(\E_{\C}(-,X),F)$ is an epimorphism and $\beta \in \Hom(F, \E_{\C}(-,Y))$ is a monomorphism. Since $\gldim (\mod \C) \leqslant 1$ and we know from Lemma \ref{lemma.2} that $\E_{\C}(-,X)$ is injective in $\mod \C$, we conclude that so is $F$. Therefore, $\beta$ is a split monomorphism, so that $F$ is a direct summand of $\E_{\C}(-,Y)$, which implies that $F$ is isomorphic to $\E_{\C}(-,V)$ for some $V \in \C$, by Theorem \ref{theorem.5}. Thus, using Theorem \ref{theorem.6}, we can conclude that the factorization $\E_{\C}(-,f) = \beta \alpha$ induces a decomposition $\E_{\C}(-,f) = \E_{\C}(-,r) \E_{\C}(-,s)$ in $\mod \C$, where $s \in \C(X,V)$ and $r \in \C(V,Y)$ are such that $\E_{\C}(-,s)$ is an epimorphism and $\E_{\C}(-,r)$ is a split monomorphism in $\mod \C$. In particular, $\E_{\C}(-,s)$ and $\E_{\C}(-,r)$ are, respectively, an epimorphism and a split monomorphism in the subcategory of $\mod \C$ consisting of the objects $\E_{\C}(-,Z)$ with $Z \in \C$. Since this subcategory is isomorphic to $\overline{\C}$ via the functor described in Theorem \ref{theorem.6}, we conclude that $\overline{f} = \overline{r} \hspace{0.15em} \overline{s}$, where $\overline{s}$ and $\overline{r}$ are, respectively, an epimorphism and a split monomorphism in $\overline{\C}$. Hence $\overline{s}$ is a $0$-cokernel of $\overline{f}$ in $\overline{\C}$.
\end{proof}

\begin{corollary}\label{corollary.4}
Let $\C$ be a $0$-abelian category with enough injectives. The following are equivalent:
\begin{enumerate}
    \item[(a)] $\overline{\C}$ is $0$-abelian.
    \item[(b)] $\overline{\C}$ is $n$-abelian for some nonnegative integer $n$.
    \item[(c)] $\overline{\C}$ has weak kernels.
\end{enumerate}
\end{corollary}

\begin{proof}
Trivially, (a) implies (b). Moreover, it is clear that a preadditive category with $n$-kernels has weak kernels when $n$ is positive, and the same claim holds for $n = 0$, as we showed in the proof of Proposition \ref{proposition.1}. Hence (b) implies (c). Below, we prove that (c) implies (a).

Given Proposition \ref{proposition.20}, we know that $\overline{\C}$ is additive, idempotent complete, and has $0$-cokernels. Therefore, by Proposition \ref{proposition.1}, $\overline{\C}$ is left coherent and $\gldim \big( \mod \overline{\C}^{\hspace{0.1em} \op} \big) \leqslant 1$. Suppose that $\overline{\C}$ has weak kernels. Then $\overline{\C}$ is right coherent, see \cite[Proposition 3.1]{2409.10438}, and it follows from \cite[Theorem A.5]{2409.10438} that $\gldim (\mod \overline{\C}) \leqslant 1$. Hence $\overline{\C}$ is $0$-abelian.
\end{proof}

For the record, let us state the dual results of Proposition \ref{proposition.20} and Corollary \ref{corollary.4}. They say that if $\C$ is a $0$-abelian category with enough projectives, then $\underline{\C}$ is additive, idempotent complete, and has $0$-kernels. Moreover, under these assumptions on $\C$, the following are equivalent:
\begin{enumerate}
    \item[(a)] $\underline{\C}$ is $0$-abelian.
    \item[(b)] $\underline{\C}$ is $n$-abelian for some nonnegative integer $n$.
    \item[(c)] $\underline{\C}$ has weak cokernels.
\end{enumerate}

Given the statements of Corollary \ref{corollary.4} and its dual version, it would be interesting to obtain conditions which guarantee that injectively stable categories and projectively stable categories have weak kernels and weak cokernels, respectively. In this direction, recall from \cite{MR591246} that a subcategory $\B$ of $\C$ is called \textit{contravariantly finite in $\C$} if for each $X \in \C$ there is a morphism $B \to X$ in $\C$ with $B \in \B$ such that every morphism $D \to X$ in $\C$ with $D \in \B$ factors through it. Dually, $\B$ is called \textit{covariantly finite in $\C$} if for each $X \in \C$ there is a morphism $X \to B$ in $\C$ with $B \in \B$ such that any other morphism $X \to D$ in $\C$ with $D \in \B$ factors through it. For example, note that $\Projectives$ is contravariantly finite in $\C$, while $\Injectives$ is covariantly finite in $\C$. The notions of contravariantly and covariantly finite subcategories are of particular interest to us due to the next result, which is well known.

\begin{proposition}\label{proposition.24}
Let $\B$ be a contravariantly finite subcategory of $\C$ that is closed under finite direct sums. If $\C$ has weak kernels, then so do $\B$ and $\C / \langle \B \rangle$.\footnote{Although we have the blanket assumption that $\C$ is additive and idempotent complete, it suffices to assume that $\C$ is additive for this proposition to be true.}
\end{proposition}

\begin{proof}
See \cite[Proposition 3.1]{MR1856980}.
\end{proof}

Note that the dual result of Proposition \ref{proposition.24} states that if $\B$ is a covariantly finite subcategory of $\C$ that is closed under finite direct sums and $\C$ has weak cokernels, then both $\B$ and $\C / \langle \B \rangle$ have weak cokernels.

\begin{corollary}\label{corollary.6}
Let $\C$ be a $0$-abelian category with enough injectives. If the subcategory $\Injectives$ of injective objects of $\C$ is contravariantly finite in $\C$, then $\overline{\C}$ is a $0$-abelian category.
\end{corollary}

\begin{proof}
By the proof of Proposition \ref{proposition.1}, $\C$ has weak kernels. Therefore, the result follows from Proposition \ref{proposition.24} and Corollary \ref{corollary.4}.
\end{proof}

By duality, Corollary \ref{corollary.6} also gives that if $\C$ is $0$-abelian, has enough projectives and its subcategory of projective objects $\Projectives$ is covariantly finite in $\C$, then $\underline{\C}$ is $0$-abelian.

In view of Corollary \ref{corollary.6} and its dual result, it would be convenient to know sufficient conditions for a subcategory to be contravariantly or covariantly finite. In the next proposition, we present one such condition, which can be found implicitly in \cite[Proposition 4.2]{MR591246} and also in \cite[Proposition 1.9]{MR816889}. For the convenience of the reader, we include a proof, which is alternative to the ones given in the references.

\begin{proposition}\label{proposition.25}
Assume that there is a commutative ring $R$ for which $\C$ is an $R$-linear category. If $X \in \C$ is such that $\C(X,Y)$ is finitely generated as an $R$-module for every $Y \in \C$, then $\add X$ is contravariantly finite in $\C$.\footnote{Actually, we prove a more general statement. Namely, that if $X \in \C$ is such that $\C(X,Y)$ is finitely generated as an $\End(X)$-module for every $Y \in \C$, then $\add X$ is contravariantly finite in $\C$. Here, $\End(X)$ stands for the endomorphism ring of $X$ in $\C$.}
\end{proposition}

\begin{proof}
Let $X \in \C$ be such that $\C(X,Y)$ is finitely generated as an $R$-module for every $Y \in \C$, and denote $\B = \add X$ and $\Lambda = \End(X)$, the endomorphism ring of $X$ in $\C$.

To begin with, we claim that $\B$ is contravariantly finite in $\C$ if and only if $\C(-,Y) \vert_{\B}$, the restriction of $\C(-,Y)$ to $\B$, is finitely generated as a $\B$-module for every $Y \in \C$.\footnote{This equivalence was observed by Auslander and Smal{\o} in \cite[page 81]{MR591246}, and it was the motivation for their choice of the terminology ``contravariantly finite subcategory''.} In fact, it follows from the Yoneda lemma that, if $B \in \B$ and $Y \in \C$, then every morphism $\B(-,B) \to \C(-,Y) \vert_{\B}$ in $\Mod \B$ is given by $\C(-,f) \vert_{\B}$, the restriction of $\C(-,f)$ to $\B$, for some $f \in \C(B,Y)$. Thus, as $\C(-,f) \vert_{\B}$ is an epimorphism in $\Mod \B$ if and only if for every $D \in \B$ and $g \in \C(D,Y)$ there is some $h \in \B(D,B)$ for which $g = fh$, the claim follows.

Now, as the evaluation functor at $X$ is an equivalence of categories $\Mod \B \to \Mod \Lambda$, see \cite[Section 8]{2409.10438}, it is easy to conclude that a given $F \in \Mod \B$ is finitely generated as a $\B$-module if and only if $F(X)$ is finitely generated as a $\Lambda$-module. In particular, for each $Y \in \C$, the $\B$-module $\C(-,Y) \vert_{\B}$ is finitely generated if and only if the $\Lambda$-module $\C(X,Y)$ is finitely generated. However, for every $Y \in \C$, we have that $\C(X,Y)$ is finitely generated as an $R$-module, and because $\Lambda$ is an $R$-algebra, we deduce that $\C(X,Y)$ is finitely generated as a $\Lambda$-module for all $Y \in \C$.
\end{proof}

Let us also note that the dual of Proposition \ref{proposition.25} says that if there is a commutative ring $R$ for which $\C$ is an $R$-linear category, and if $Y \in \C$ is such that $\C(X,Y)$ is finitely generated as an $R$-module for every $X \in \C$, then $\add Y$ is covariantly finite in $\C$.

In Section \ref{section.5}, we will give an example of a $0$-abelian category to which Corollary \ref{corollary.6}, Proposition \ref{proposition.25} and their duals can be applied, which shows that its injectively and projectively stable categories are $0$-abelian.

While Proposition \ref{proposition.20} and Corollaries \ref{corollary.4} and \ref{corollary.6} provide a fair understanding about the $n$-abelianness of the stable categories of a $0$-abelian category, it is natural to ask if these results admit analogues for the stable categories of an abelian category. Below, we present possible analogues, whose proofs follow closely the arguments used for their $0$-abelian versions.


\begin{proposition}\label{proposition.23}
Let $\A$ be an abelian category with enough injectives for which there is a positive integer $d$ such that $\gldim \A \leqslant d$. Then $\overline{\A}$ is additive, idempotent complete, and has $(3d-2)$-cokernels.
\end{proposition}

\begin{proof}
To begin with, we remark that $\Ext_{\A}^{1}(-,X)$ is a finitely presented $\A$-module for each $X \in \A$. Indeed, by considering the long exact sequence induced by a short exact sequence \[ \begin{tikzcd}
0 \arrow[r] & X \arrow[r] & I \arrow[r] & \Sigma X \arrow[r] & 0
\end{tikzcd} \] in $\A$ with $I$ injective, we obtain a projective presentation (resolution) of $\Ext_{\A}^{1}(-,X)$ in $\mod \A$.\footnote{Recall from Theorem \ref{theorem.1} that $\mod \A$ is an abelian category with $\gldim (\mod \A) \leqslant 2$.} In particular, we conclude from Proposition \ref{proposition.6} that $\Ext_{\A}^{1}(-,X) \in \eff \A$ for every $X \in \A$. Moreover, by the Hilton-Rees theorem, see \cite{MR124377} or \cite[Theorem 8]{MR3516080}, for each $X,Y \in \A$ there is an isomorphism of abelian groups \[ \overline{\A}(X,Y) \to \Hom(\Ext_{\A}^{1}(-,X), \Ext_{\A}^{1}(-,Y)) \] given by $\overline{f} \mapsto \Ext_{\A}(-,f)$, which natural in both $X$ and $Y$. Hence these isomorphisms define a fully faithful additive functor $\overline{\A} \to \mod \A$. Furthermore, since $\A$ has finite global dimension, it follows from \cite[Proposition 4.7]{MR0212070}, or \cite{MR237606}, that every direct summand of $\Ext_{\A}^{1}(-,X)$ in $\mod \A$ is isomorphic to $\Ext_{\A}^{1}(-,V)$ for some $V \in \A$. Also, recall that $\eff \A$ is \textit{Serre subcategory} of $\mod \A$, that is, for every short exact sequence \[ \begin{tikzcd}
0 \arrow[r] & F \arrow[r] & G \arrow[r] & H \arrow[r] & 0
\end{tikzcd} \] in $\mod \A$, it holds that $G \in \eff \A$ if and only if $F,H \in \eff \A$, see \cite[Section 3]{MR0212070}. Consequently, $\eff \A$ is an abelian category and a sequence in $\eff \A$ is exact in $\eff \A$ if and only if it is exact in $\mod \A$. Finally, it follows from \cite[Proposition 4.3]{MR0212070} that $\Ext_{\A}^{1}(-,X)$ is injective in $\eff \A$ for every $X \in \A$.

Now, we are ready to prove the proposition. First, note that $\overline{\A}$ is additive since $\A$ is additive. Second, let $\Ext_{\A}^{1}(-,\A)$ be the subcategory of $\mod \A$ consisting of the objects $\Ext_{\A}^{1}(-,X)$ with $X \in \A$. As we have remarked in the above paragraph, $\Ext_{\A}^{1}(-,\A)$ is closed under direct summands in $\mod \A$, hence it is easy to deduce that it is idempotent complete. Thus, because the functor $\overline{\A} \to \mod \A$ mentioned above gives an isomorphism between $\overline{\A}$ and $\Ext_{\A}^{1}(-,\A)$, we conclude that $\overline{\A}$ is idempotent complete, see also \cite{MR206069}. Third, let $\overline{f} \in \overline{\A}(X,Y)$ be a morphism in $\overline{\A}$, represented by some $f \in \A(X,Y)$. Since $\A$ has enough injectives, there is a monomorphism $g \in \A(Y,I)$ in $\A$ with $I$ injective. Therefore, we can consider a short exact sequence \[ \begin{tikzcd}[ampersand replacement=\&]
0 \arrow[r] \& X \arrow[r, "{\left(\begin{smallmatrix}
        f \\ g
    \end{smallmatrix}\right)}"] \&[0.4em] Y \oplus I \arrow[r] \& Z \arrow[r] \& 0
\end{tikzcd} \] in $\A$, and its induced long exact sequence gives an exact sequence \[ \begin{tikzcd}
\Ext_{\A}^{1}(-,X) \arrow[r, "{\Ext_{\A}^{1}(-,f)}"] &[2.4em] \Ext_{\A}^{1}(-,Y) \arrow[r] & \Ext_{\A}^{1}(-,Z) \arrow[out=0, in=180, looseness=2, overlay, lld] & {} \\[0.4em]
\Ext_{\A}^{2}(-,X) \arrow[r]             & \Ext_{\A}^{2}(-,Y) \arrow[r]      & \Ext_{\A}^{2}(-,Z) \arrow[out=0, in=180, looseness=2, overlay, lld] & {} \\[0.4em]
{\phantom{\Ext_{\A}^{2}(-,X)}}                       & \cdots               & {\phantom{\Ext_{\A}^{2}(-,Z)}} \arrow[out=0, in=180, looseness=2, overlay, lld] & {} \\[0.4em]
\Ext^{d}_{\A}(-,X) \arrow[r] & \Ext_{\A}^{d}(-,Y) \arrow[r] & \Ext_{\A}^{d}(-,Z) \arrow[r] & 0
\end{tikzcd} \] in $\mod \A$, which is also an exact sequence in $\eff \A$. In this case, given that $\Ext_{\A}^{j}(-,V) \simeq \Ext_{\A}^{1}(-,\Sigma^{j-1}V)$ for every $V \in \A$ and $j \geqslant 2$, where $\Sigma^{j-1} V$ is a $(j-1)$\textsuperscript{th} cosyzygy of $V$,\footnote{A $j$\textsuperscript{th} \textit{cosyzygy} of an object in $\A$ is the dual notion of a $j$\textsuperscript{th} syzygy, defined in \cite[Subsection 2.4]{2409.10438}.} the above sequence is isomorphic to a sequence that lies in the category $\Ext_{\A}^{1}(-,\A)$, which gives a $(3d-2)$-cokernel of $\Ext_{\A}^{1}(-,f)$ in $\Ext_{\A}^{1}(-,\A)$. In fact, given $V \in \A$, because $\Ext_{\A}^{1}(-,V)$ is injective in $\eff \A$, when we apply $\Hom\big(-,\Ext_{\A}^{1}(-,V)\big)$ to the mentioned sequence in $\Ext_{\A}^{1}(-,\A)$, which is exact in $\eff \A$, we obtain an exact sequence in $\Ab$. Therefore, by considering the isomorphism between $\overline{\A}$ and $\Ext_{\A}^{1}(-,\A)$ induced by the functor $\overline{\A} \to \mod \A$, we conclude that $\overline{f}$ has a $(3d-2)$-cokernel in $\overline{\A}$.
\end{proof}

Recall from \cite[Section 3]{2409.10438} that we call a category \textit{pre-$n$-abelian} if it is additive, idempotent complete, and has both $n$-kernels and $n$-cokernels. By Proposition \ref{proposition.1}, a category is pre-$0$-abelian if and only if it is $0$-abelian.

\begin{corollary}\label{corollary.5}
Let $\A$ be an abelian category with enough injectives for which there is a positive integer $d$ such that $\gldim \A \leqslant d$. The following are equivalent:
\begin{enumerate}
    \item[(a)] $\overline{\A}$ is pre-$(3d-2)$-abelian.
    \item[(b)] $\overline{\A}$ is pre-$n$-abelian for some nonnegative integer $n$.
    \item[(c)] $\overline{\A}$ has weak kernels.
\end{enumerate}
\end{corollary}

\begin{proof}
This is similar to the proof of Corollary \ref{corollary.4}, the only difference being that one should use \cite[Proposition 3.2]{2409.10438} or \cite[Proposition 3.3]{2409.10438} in place of Proposition \ref{proposition.1} to prove that (c) implies (a).
\end{proof}

\begin{corollary}\label{corollary.7}
Let $\A$ be an abelian category with enough injectives for which there is a positive integer $d$ such that $\gldim \A \leqslant d$. If the subcategory of injective objects of $\A$ is contravariantly finite in $\A$, then $\overline{\A}$ is a pre-$(3d-2)$-abelian category.
\end{corollary}

\begin{proof}
Follows from Proposition \ref{proposition.24} and Corollary \ref{corollary.5}.
\end{proof}

Let us also state the duals of Proposition \ref{proposition.23} and Corollaries \ref{corollary.5} and \ref{corollary.7}. Namely, if $\A$ is an abelian category with enough projectives for which there is a positive integer $d$ such that $\gldim \A \leqslant d$, then $\underline{\A}$ is additive, idempotent complete, and has $(3d-2)$-kernels. Moreover, assuming these conditions on $\A$, the following are equivalent:
\begin{enumerate}
    \item[(a)] $\underline{\A}$ is pre-$(3d-2)$-abelian.
    \item[(b)] $\underline{\A}$ is pre-$n$-abelian for some nonnegative integer $n$.
    \item[(c)] $\underline{\A}$ has weak cokernels.
\end{enumerate}
Also, if in addition to the previous assumptions on $\A$, the subcategory of projective objects of $\A$ is covariantly finite in $\A$, then $\underline{\A}$ is is a pre-$(3d-2)$-abelian category.

Given the previous results, it is natural to ask under which conditions on $\A$, in the settings of Proposition \ref{proposition.23} and its dual, respectively, the categories $\overline{\A}$ and $\underline{\A}$ are $(3d-2)$-abelian, when $d$ is taken to be as small as possible.\footnote{That is, when $d = \gldim \A$, excluding the case when $\A$ has global dimension equal to zero.} Below, we indicate a few results in the literature concerning this problem for $d = 1$, which may shed some light on the general case.

Let $\A$ be an abelian category with enough projectives and $\gldim \A \leqslant 1$. Denote the subcategory of projective objects of $\A$ by $\Projectives (\A)$, and let $^{\perp} \Projectives (\A)$ be the subcategory of $\A$ consisting of the objects $X \in \A$ such that $\A(X,P) = 0$ for every $P \in \Projectives (\A)$. Consider the following conditions:
\begin{enumerate}
    \item[(a)] $\Projectives (\A)$ is covariantly finite in $\A$.
    \item[(b)] $\underline{\A}$ is pre-abelian.
    \item[(c)] $\underline{\A}$ is abelian.
    \item[(d)] $\underline{\A}$ is balanced.
    \item[(e)] $^{\perp} \Projectives (\A)$ is closed under subobjects in $\A$.
    \item[(f)] $\domdim \A \geqslant 1$.
\end{enumerate}
If $\A$ satisfies (a), then we know from the dual of Corollary \ref{corollary.7} that (b) holds. Furthermore, it follows from \cite[Proposition 9 and page 564]{MR2563182} that if $\A$ satisfies (a), then conditions (c), (d) and (e) are equivalent, and they are also equivalent to (f) if, in addition to satisfying (a), the category $\A$ has enough injectives. Now, if $\A = \Mod \Lambda$ for some ring $\Lambda$ (and $\gldim \A \leqslant 1$), then we conclude from \cite[Theorem 9.5]{MR3335998} and \cite[Theorem 3.2]{MR269686} that (c) is equivalent to (f), and also that (c) and (f) imply (a), by \cite[Proposition 9.2]{MR3335998} and \cite[Proposition 3.5]{MR1244922} or \cite[Proposition 3.14]{MR1248754} or \cite[Theorem 5.7]{MR1856980}. Note that, by taking opposite categories, we deduce that dual results also hold for abelian categories with enough injectives and global dimension at most $1$, except, \textit{a priori}, the ``duals'' of results about $\Mod \Lambda$ for a ring $\Lambda$, since the opposite category of $\Mod \Lambda$ is no longer a category of modules over a ring.\footnote{The reader might want to compare, for example, \cite[Theorem 5.7]{MR1856980} with \cite[Theorem 5.3]{MR1856980}.}

The author believes that it would be a fruitful project to investigate further the relations between the above conditions for an abelian category $\A$. It would be even more interesting to consider the general case where $\gldim \A \leqslant d$ for a positive integer $d$, and to study when the injectively and projectively stable categories of $\A$ are $(3d-2)$-abelian, or possibly $n$-abelian for some $n$ not of the form $3d - 2$.

It seems appropriate to end this discussion by pointing out where one could look for examples. Let $\Lambda$ be a ring and $d$ be a positive integer. It follows from \cite[Lemma 5.2]{MR1856980} and \cite[Proposition 8.3]{2409.10438} that if $\Lambda$ is right coherent, then $\proj \Lambda$ is covariantly finite in $\mod \Lambda$ if and only if $\Lambda$ is left coherent, see also \cite[Corollary 3.11]{MR1248754}. Consequently, if $\Lambda$ is coherent and has weak dimension at most $d$, so that $\gldim (\mod \Lambda) \leqslant d$, see \cite[Proposition 1.1]{MR306265}, then we conclude from the dual of Corollary \ref{corollary.7} that $\underline{\mod} \Lambda$, the projectively stable category of $\mod \Lambda$, is a pre-$(3d-2)$-abelian category. More generally, it follows from \cite[Lemma 5.2]{MR1856980} and the dual of \cite[Propositions 3.1]{2409.10438} that if $\C$ is right coherent, then $\proj \C$ is covariantly finite in $\mod \C$ if and only if $\C$ is left coherent. Therefore, we can conclude from \cite[Propositions 3.2]{2409.10438} and the dual of Corollary \ref{corollary.7} that if $\C$ is a pre-$n$-abelian category, then the projectively stable category $\underline{\mod} \C$ of $\mod \C$ is pre-$(3n+1)$-abelian. In particular, we deduce from Theorem \ref{theorem.2} and the previous discussion concerning \cite[Proposition 9 and page 564]{MR2563182} that if $\C$ is a $0$-abelian category with enough injectives, then $\underline{\mod} \C$ is abelian. Now, getting back to the ring case, as we implicitly mentioned before, the category of projective objects in $\Mod \Lambda$ is covariantly finite in $\Mod \Lambda$ if and only if $\Lambda$ is left coherent and right perfect, see \cite[Proposition 3.5]{MR1244922} or \cite[Proposition 3.14]{MR1248754} or \cite[Theorem 5.7]{MR1856980}. Consequently, if $\Lambda$ is left coherent, right perfect and has right global dimension at most $d$, then the projectively stable category $\underline{\Mod} \Lambda$ of $\Mod \Lambda$ is pre-$(3d-2)$-abelian, by the dual of Corollary \ref{corollary.7}. Finally, by \cite[Theorem 5.3]{MR1856980} or \cite[Theorem 5.4.1]{MR1753146}, the subcategory of injective objects of $\Mod \Lambda$ is contravariantly finite in $\Mod \Lambda$ if and only if $\Lambda$ is right noetherian. Hence we deduce from Corollary \ref{corollary.7} that if $\Lambda$ is right noetherian and has right global dimension at most $d$, then the injectively stable category $\overline{\Mod} \Lambda$ of $\Mod \Lambda$ is pre-$(3d-2)$-abelian.\footnote{In the same way that we generalized the statements about $\proj \Lambda$ being covariantly finite in $\mod \Lambda$ to when $\proj \C$ is covariantly finite in $\mod \C$, it is also possible to generalize the statements about the covariant and contravariant finiteness of the subcategories of, respectively, projective and injective objects of $\Mod \Lambda$. We refer the interested reader to \cite[Theorem 7.6]{MR2027559} for such generalizations.}

\section{\texorpdfstring{$0$}{0}-Abelian categories with additive generators}\label{section.5}

Similarly to \cite[Section 8]{2409.10438}, in this section, we consider previous results for the case of additive and idempotent complete categories with additive generators. Since these are precisely the categories that are equivalent to $\proj \Lambda$ for some ring $\Lambda$, see \cite[Lemma 8.2]{2409.10438}, our task in this section amounts to specializing previous results to rings and modules over rings. This is achieved through the equivalences of categories $\Mod (\proj \Lambda) \approx \Mod \Lambda$ and $\Mod (\proj \Lambda)^{\op} \approx \Mod \Lambda^{\op}$, both given by the evaluation at $\Lambda$, which also induce equivalences of categories $\mod (\proj \Lambda) \approx \mod \Lambda$ and $\mod (\proj \Lambda)^{\op} \approx \mod \Lambda^{\op}$, see \cite[Section 8]{2409.10438}. As we show, this approach leads us to conclude that $0$-abelian categories with additive generators are in correspondence with ``semi-hereditary'' rings. Therefore, examples of $0$-abelian categories can be constructed from such rings. Moreover, using this approach, we deduce a few results on ``semi-hereditary'' rings, which should be well known, although the proofs we present seem to be new.

For the rest of this section, let $\Lambda$ be a ring.

\subsection{How to construct examples}

Recall that a ring $\Lambda$ is called \textit{right semi-hereditary} when every finitely generated submodule of a projective right $\Lambda$-module is projective, see \cite[Corollary 2.30]{MR1653294}. Also, $\Lambda$ is called \textit{left semi-hereditary} when every finitely generated submodule of a projective left $\Lambda$-module is projective. If $\Lambda$ is both right and left semi-hereditary, then we say that $\Lambda$ is \textit{semi-hereditary}. The next proposition explains why such rings are of \mbox{interest to us.}


\begin{proposition}\label{proposition.26}
Let $\Lambda$ be a ring. The following are equivalent:
\begin{enumerate}
    \item[(a)] $\Lambda$ is right semi-hereditary.
    \item[(b)] $\Lambda$ is right coherent and has weak dimension at most $1$.
    \item[(c)] $\Lambda$ is right coherent and $\gldim (\mod \Lambda) \leqslant 1$.
\end{enumerate}
\end{proposition}

\begin{proof}
If $\Lambda$ is right coherent, then $\gldim (\mod \Lambda)$ coincides with the weak dimension of $\Lambda$, see \cite[Proposition 1.1]{MR306265}. Hence (b) is equivalent to (c). Moreover, it is easy to see that (a) is equivalent to (c) since $\Lambda$ is right coherent if and only if every finitely generated right ideal of $\Lambda$ is finitely presented,\footnote{Recall from \cite{2409.10438} that we have defined a ring $\Lambda$ to be \textit{right coherent} when the category $\mod \Lambda$ is abelian. It is well known that this is the case if and only if every finitely generated right ideal of $\Lambda$ is finitely presented. This can be proved with the assistance of \cite[Corollary 4.52]{MR1653294}.} and $\Lambda$ is right semi-hereditary if and only if every finitely generated right ideal of $\Lambda$ is projective, see \mbox{\cite[Corollary 2.30]{MR1653294}.}
\end{proof}


Observe that, by taking $\Lambda^{\op}$ in place of $\Lambda$ in Proposition \ref{proposition.26}, we obtain the corresponding result for left $\Lambda$-modules, which characterizes left semi-hereditary rings.

We can now start to specialize previous results to rings and modules over rings.

\begin{proposition}\label{proposition.5}
Let $\Lambda$ be a ring. The category $\proj \Lambda$ has $0$-kernels if and only if $\Lambda$ is right semi-hereditary.
\end{proposition}

\begin{proof}
Follows from Propositions \ref{proposition.1} and \ref{proposition.26} and the equivalence $\mod (\proj \Lambda) \approx \mod \Lambda$.
\end{proof}

By considering Proposition \ref{proposition.1}, the left version of Proposition \ref{proposition.26} and the equivalence $\mod (\proj \Lambda)^{\op} \approx \mod \Lambda^{\op}$, we also conclude that $\proj \Lambda$ has $0$-cokernels if and only if $\Lambda$ is left semi-hereditary. This is, of course, the dual statement of Proposition \ref{proposition.5}.\footnote{Indeed, note that $(\proj \Lambda)^{\op}$ is equivalent to $\proj \Lambda^{\op}$ via the duality $(-)^{\ast} : \proj \Lambda \leftrightarrow \proj \Lambda^{\op}$, see \cite[Section 8]{2409.10438} for details. We will mention this fact again right after Proposition \ref{proposition.29}.}

\begin{theorem}\label{theorem.7}
Let $\Lambda$ be a ring. The category $\proj \Lambda$ is $0$-abelian if and only if $\Lambda$ is semi-hereditary.
\end{theorem}

\begin{proof}
Given our definition of a $0$-abelian category, the result follows from the equivalences $\mod (\proj \Lambda) \approx \mod \Lambda$ and $\mod (\proj \Lambda)^{\op} \approx \mod \Lambda^{\op}$ and Proposition \ref{proposition.26} and its left version. It also follows from Theorem \ref{theorem.3} and Proposition \ref{proposition.5} and its dual.
\end{proof}

Due to Theorem \ref{theorem.7}, we can now conclude that $0$-abelian categories with additive generators are in correspondence with semi-hereditary rings. The precise statement of this result is given below, which is the case $n = 0$ of \cite[Theorem 8.13]{2409.10438}.

\begin{theorem}\label{theorem.9}
There is a bijective correspondence between the equivalence classes of $0$-abelian categories with additive generators and the Morita equivalence classes of semi-hereditary rings. The correspondence is given as follows:
\begin{enumerate}
    \item[(a)] If $\C$ is a $0$-abelian category with an additive generator, then send it to $\End(X)$, the endomorphism ring of $X$, where $X$ is an additive generator of $\C$.
    \item[(b)] If $\Lambda$ is a semi-hereditary ring, then send it to $\proj \Lambda$.
\end{enumerate}
\end{theorem}

\begin{proof}
Follows from \cite[Lemma 8.2]{2409.10438} and Theorem \ref{theorem.7}.
\end{proof}

In particular, Theorem \ref{theorem.9} says that all examples of $0$-abelian categories with additive generators are of the form $\proj \Lambda$ with $\Lambda$ a semi-hereditary ring. For instance, we might take $\Lambda = \mathbb{Z}$, the ring of integers, and conclude that $\proj \mathbb{Z}$ is a $0$-abelian category. Explicitly, $\proj \mathbb{Z}$ is the subcategory of $\Ab$ whose objects are finite direct sums of copies of $\mathbb{Z}$ (including the empty direct sum, which gives the zero object). Observe that every nonzero morphism $\mathbb{Z} \to \mathbb{Z}$ is a bimorphism in $\proj \mathbb{Z}$, but not an isomorphism, unless $\mathbb{Z} \to \mathbb{Z}$ is the identity. Hence $\mathbb{Z}$ is neither injective nor projective in $\proj \mathbb{Z}$. Consequently, $\proj \mathbb{Z}$ is a $0$-abelian category with no nonzero injective object and no nonzero projective object.

Motivated by the above example, we explain next how to describe the injective and the projective objects in the category $\proj \Lambda$ when $\proj \Lambda$ is $0$-abelian. To do this, let us begin with the more general case when $\proj \Lambda$ is right or left comprehensive.

Recall from \cite{2409.10438} that $\Lambda$ is called \textit{right comprehensive}, \textit{left comprehensive} and \textit{comprehensive} when $\proj \Lambda$ is right comprehensive, left comprehensive and comprehensive, respectively. Since these definitions are not widely known, let us elaborate on them.

We know from \cite[Proposition 7.1]{2409.10438} that if $P \in \proj \Lambda$ is such that $(\proj \Lambda)(-,P)$ is injective in $\mod (\proj \Lambda)$, then $P$ is injective in $\proj \Lambda$. Therefore, by using the equivalence $\mod (\proj \Lambda) \approx \mod \Lambda$, we conclude that if $P \in \proj \Lambda$ is such that $P$ is injective in $\mod \Lambda$, then $P$ is injective in $\proj \Lambda$. One could then ask when the converse of this statement holds, in which case the injective objects of $\proj \Lambda$ coincide with the projective injective objects of $\mod \Lambda$. According to our definitions, the rings $\Lambda$ for which such converse holds are the right comprehensive rings, as we point out below.

\begin{proposition}\label{proposition.29}
A ring $\Lambda$ is right comprehensive if and only if every injective object in $\proj \Lambda$ is injective in $\mod \Lambda$.
\end{proposition}

\begin{proof}
Follows from the equivalence $\mod (\proj \Lambda) \approx \mod \Lambda$.
\end{proof}

Before we give the dual discussion, let us first remark that, similarly to the contravariant functor $(-)^{\ast} : \mod \C \to \Mod \C^{\op}$, there is also a contravariant functor $(-)^{\ast} : \mod \Lambda \to \Mod \Lambda^{\op}$, which is given by $M^{\ast} = \Hom_{\Lambda}(M,\Lambda)$ for each $M \in \mod \Lambda$ and $f^{\ast} = \Hom_{\Lambda}(f,\Lambda)$ for each morphism $f$ in $\mod \Lambda$. We call $M^{\ast}$ the \textit{dual} of $M$. Note that the latter functor can be obtained from the former by taking $\C = \proj \Lambda$ and considering the evaluation functors at $\Lambda$. Likewise, the contravariant functor $(-)^{\ast} : \mod \C^{\op} \to \Mod \C$ induces a contravariant functor $(-)^{\ast} : \mod \Lambda^{\op} \to \Mod \Lambda$ when $\C = \proj \Lambda$. Finally, let us mention that, in the same way that there is a duality $(-)^{\ast} : \proj \C \leftrightarrow \proj \C^{\op}$ induced by the previous functors on $\C$-modules,\footnote{This is a consequence of the Yoneda lemma.} the corresponding functors for $\Lambda$-modules also induce a duality $(-)^{\ast} : \proj \Lambda \leftrightarrow \proj \Lambda^{\op}$.

Now, we know from the dual of \cite[Proposition 7.1]{2409.10438} that if $P \in \proj \Lambda$ is such that $(\proj \Lambda)(P,-)$ is injective in $\mod (\proj \Lambda)^{\op}$, then $P$ is projective in $\proj \Lambda$. Thus, by using the equivalence $\mod (\proj \Lambda)^{\op} \approx \mod \Lambda^{\op}$, we conclude that if $P \in \proj \Lambda$ is such that $P^{\ast}$ is injective in $\mod \Lambda^{\op}$, then $P$ is projective in $\proj \Lambda$. Moreover, as we indicate in the next proposition, the rings $\Lambda$ for which the converse of this statement holds are the left comprehensive rings. Consequently, when $\Lambda$ is left comprehensive, the projective objects of $\proj \Lambda$ are given by the finitely generated projective $\Lambda$-modules whose duals are injective in $\mod \Lambda^{\op}$, which can be identified with the projective injective objects of $\mod \Lambda^{\op}$ via the duality $(-)^{\ast} : \proj \Lambda \leftrightarrow \proj \Lambda^{\op}$.

\begin{proposition}\label{proposition.30}
A ring $\Lambda$ is left comprehensive if and only if the dual of every projective object in $\proj \Lambda$ is injective in $\mod \Lambda^{\op}$.
\end{proposition}

\begin{proof}
Follows from the equivalence $\mod (\proj \Lambda)^{\op} \approx \mod \Lambda^{\op}$.
\end{proof}

Given the above discussion around Propositions \ref{proposition.29} and \ref{proposition.30}, it is convenient to recall from \cite{MR258888} that a $\Lambda$-module $M$ is called \textit{FP-injective} if $\Ext_{\Lambda}^{1}(N,M) = 0$ for every finitely presented $\Lambda$-module $N$. Observe that such modules were also called ``absolutely pure'' in \cite{MR224649} and \cite{MR294409}. As we remarked in \cite[Section 8]{2409.10438}, if $\Lambda$ is right coherent, then an object in $\mod \Lambda$ is injective in $\mod \Lambda$ if and only if it is FP-injective. Hence we have:

\begin{proposition}\label{proposition.31}
Let $\Lambda$ be a ring that is right coherent and right comprehensive. Then the injective objects of $\proj \Lambda$ are precisely the finitely generated projective $\Lambda$-modules which are FP-injective.
\end{proposition}

\begin{proof}
Follows from Proposition \ref{proposition.29} and the paragraph preceding it.
\end{proof}

Also, the dual of Proposition \ref{proposition.31} is given by:

\begin{proposition}\label{proposition.32}
Let $\Lambda$ be a ring that is left coherent and left comprehensive. Then the projective objects of $\proj \Lambda$ are precisely the finitely generated projective $\Lambda$-modules whose duals are FP-injective.
\end{proposition}

\begin{proof}
Follows from Proposition \ref{proposition.30} and the paragraph preceding it. Alternatively, we can deduce it by replacing $\Lambda$ by $\Lambda^{\op}$ in Proposition \ref{proposition.31} and by using the duality $(-)^{\ast} : \proj \Lambda \leftrightarrow \proj \Lambda^{\op}$.
\end{proof}

Now, let us recall our goal of describing the injective and the projective objects in a $0$-abelian category of the form $\proj \Lambda$. In this direction, we only need one more result:

\begin{proposition}\label{proposition.33}
Every semi-hereditary ring is comprehensive.
\end{proposition}

\begin{proof}
Follows from Theorem \ref{theorem.7} and Proposition \ref{proposition.3}.
\end{proof}

To conclude, suppose that $\proj \Lambda$ is $0$-abelian. Since Theorem \ref{theorem.7} says that $\Lambda$ is semi-hereditary, we get from Proposition \ref{proposition.26} (and its left version) and Proposition \ref{proposition.33} that $\Lambda$ is coherent and comprehensive. Therefore, the injective and the projective objects of $\proj \Lambda$ are described by Propositions \ref{proposition.31} and \ref{proposition.32}. More generally, these results describe the injective and the projective objects of $\proj \Lambda$ whenever $\proj \Lambda$ is an $n$-abelian category, by \cite[Proposition 7.3 and Theorem 8.5]{2409.10438}.

It would be natural to ask next when a $0$-abelian category of the form $\proj \Lambda$ has enough injectives or projectives. Let us answer this question.

\begin{theorem}\label{theorem.8}
Let $\Lambda$ be a ring. The following are equivalent:
\begin{enumerate}
    \item[(a)] $\proj \Lambda$ is $0$-abelian and has enough injectives.
    \item[(b)] $\Lambda$ is coherent and $\gldim (\mod \Lambda) \leqslant 1 \leqslant \domdim (\mod \Lambda)$.
    \item[(c)] $\Lambda$ is semi-hereditary and, as a right $\Lambda$-module, it embeds into a finitely generated projective right $\Lambda$-module which is FP-injective.\footnote{That is, $\Lambda$ is semi-hereditary and there is a monomorphism of right $\Lambda$-modules from $\Lambda$ to a finitely generated projective FP-injective $\Lambda$-module. When $\Lambda$ is right noetherian, the latter condition is equivalent to the injective envelope of $\Lambda$ (viewed as a right $\Lambda$-module) being projective, which is also equivalent to $\Lambda$ being right artinian and ``right QF-3''. This follows from \cite[paragraph after Proposition 8.7]{2409.10438}, \cite[Lemma 1]{MR233854} and \cite[Theorem A]{MR268222}, see also \cite[Lemma 6.1 and Proposition 6.3]{MR349740}.}
\end{enumerate}
Moreover, if the above conditions are satisfied, then $\mod \Lambda$ has enough injectives.
\end{theorem}

\begin{proof}
It follows from Theorem \ref{theorem.2} and the equivalence $\mod (\proj \Lambda) \approx \mod \Lambda$ that (a) and (b) are equivalent, and also that if these conditions are satisfied, then $\mod \Lambda$ has enough injectives. Furthermore, we can verify that (b) and (c) are equivalent by using Proposition \ref{proposition.26} (and its left version) and the fact that, when $\Lambda$ is right coherent, an object in $\mod \Lambda$ is injective in $\mod \Lambda$ if and only if it is FP-injective.
\end{proof}

By taking $\Lambda^{\op}$ in place of $\Lambda$ in Theorem \ref{theorem.8} and using the fact that $(\proj \Lambda)^{\op}$ is equivalent to $\proj \Lambda^{\op}$, we also conclude that the following are equivalent:
\begin{enumerate}
    \item[(a)] $\proj \Lambda$ is $0$-abelian and has enough projectives.
    \item[(b)] $\Lambda$ is coherent and $\gldim (\mod \Lambda^{\op}) \leqslant 1 \leqslant \domdim (\mod \Lambda^{\op})$.
    \item[(c)] $\Lambda$ is semi-hereditary and, as a left $\Lambda$-module, it embeds into a finitely generated projective left $\Lambda$-module which is FP-injective.\footnote{When $\Lambda$ is left noetherian, the latter condition occurs if and only if the injective envelope of $\Lambda$ (viewed as a left $\Lambda$-module) is projective if and only if $\Lambda$ is left artinian and ``left QF-3'', as in the previous footnote.}
\end{enumerate}
Moreover, $\mod \Lambda^{\op}$ has enough injectives when the above conditions are satisfied.

At this moment, the reader might want to see an example where the previous results can be applied. We present such an example below.

Fixed a positive integer $m$ and a field $K$, let $\Lambda_{m}$ be the ring of lower triangular $m \times m$ matrices with entries in $K$, which is isomorphic to the path algebra $K \mathbb{A}_{m}$ of the quiver $\mathbb{A}_{m}$ over $K$, see \cite{MR2197389}. Since $\Lambda_{m}$ is a finite dimensional $K$-algebra, it is both right and left artinian, hence the dominant dimensions of the categories $\mod \Lambda_{m}$ and $\mod ({\Lambda_{m}}^{\op})$ coincide, as we remarked in \cite[paragraph after Corollary 8.10]{2409.10438}. Moreover, their global dimensions also agree since $\Lambda_{m}$ is coherent, see \cite[Theorem A.5]{2409.10438}. Therefore, as \[ \gldim (\mod \Lambda_{m}) \leqslant 1 \leqslant \domdim (\mod \Lambda_{m}), \] we conclude from Theorem \ref{theorem.8} and its dual that $\proj \Lambda_{m}$ is a $0$-abelian category with enough injectives and enough projectives. In particular, $\Lambda_{m}$ is comprehensive, by Proposition \ref{proposition.3}. Let us describe the objects of this category. For each $1 \leqslant i \leqslant m$, let $E_{ii} \in \Lambda_{m}$ be the matrix whose only nonzero entry is $1$ at the $i$\textsuperscript{th} row and $i$\textsuperscript{th} column, and consider $P_{i} = E_{ii} \Lambda_{m}$, which is an  indecomposable $\Lambda_{m}$-module. Since there is a direct sum decomposition $\Lambda_{m} \simeq P_{1} \oplus \cdots \oplus P_{m}$, we conclude that $\proj \Lambda_{m}$ is the subcategory of $\mod \Lambda_{m}$ given by finite direct sums of copies of $P_{1}, \ldots, P_{m}$. Now, because $P_{m}$ is the only injective $\Lambda_{m}$-module among $P_{1}, \ldots, P_{m}$, it follows from Proposition \ref{proposition.31} that the subcategory of injective objects of $\proj \Lambda_{m}$ is given by $\add P_{m}$. On the other hand, as $P_{1}$ is the only $\Lambda_{m}$-module among $P_{1}, \ldots, P_{m}$ whose dual is injective, we obtain from Proposition \ref{proposition.32} that the subcategory of projective objects of $\proj \Lambda_{m}$ is given by $\add P_{1}$.

We can continue the discussion on the above example a bit further. In fact, recall from Corollary \ref{corollary.1} that every subcategory of a $0$-abelian category which is closed under finite direct sums and direct summands is $0$-abelian. So, we could ask what kind of $0$-abelian categories are obtained this way from the $0$-abelian category $\proj \Lambda_{m}$ given above. Well, keeping the same notation as before, note that every nonzero subcategory of $\proj \Lambda_{m}$ that is closed under finite direct sums and direct summands is of the form $\add X_{k}$, where $X_{k} = P_{\alpha_{1}} \oplus \cdots \oplus P_{\alpha_{k}}$ for some integer $1 \leqslant k \leqslant m$ and some integers $1 \leqslant \alpha_{1} < \cdots < \alpha_{k} \leqslant m$. However, such a subcategory is equivalent to $\proj \End (X_{k})$, see \cite[Section 8]{2409.10438}, and it is easy to see that the endomorphism ring $\End (X_{k})$ is isomorphic to $\Lambda_{k}$, hence $\add X_{k}$ is equivalent to $\proj \Lambda_{k}$. Now, in another direction, observe that, by Proposition \ref{proposition.25} and its dual statement, $\add X$ is contravariantly and covariantly finite in $\proj \Lambda_{m}$ for every $X \in \proj \Lambda_{m}$. Therefore, it follows from Corollary \ref{corollary.6} and its dual result that both the injectively and the projectively stable categories of $\proj \Lambda_{m}$ are $0$-abelian. Let us describe these categories. As we saw before, the subcategory of injective objects of $\proj \Lambda_{m}$ is $\add P_{m}$, while the subcategory of projective objects is $\add P_{1}$. If $m = 1$, the stable categories of $\proj \Lambda_{m}$ are just zero, so assume that $m \geqslant 2$. It is not difficult to verify that $\proj \Lambda_{m} / \langle \add P_{m} \rangle$ is equivalent to $\add (P_{1} \oplus \cdots \oplus P_{m-1})$ and that $\proj \Lambda_{m} / \langle \add P_{1} \rangle$ is equivalent to $\add (P_{2} \oplus \cdots \oplus P_{m})$. Hence we conclude from the previous description of $\add X_{k}$ that the injectively and the projectively stable categories of $\proj \Lambda_{m}$ are both equivalent to $\proj \Lambda_{m-1}$.

\subsection{Applications to semi-hereditary rings}

We now present a few applications of the theory of $0$-abelian categories to semi-hereditary rings. Although the results that we present in this subsection should be well known, our proofs seem to be new.

Below, we show a result that was independently proved by Lenzing in \cite[Satz 1]{MR280520} and by Colby and Rutter in \cite[Theorem 2.10]{MR269686}, see also \cite[Corollary 3]{MR859386}.

\begin{proposition}\label{proposition.27}
If $\Lambda$ is a right semi-hereditary ring, then the endomorphism ring of every finitely generated projective right $\Lambda$-module is right semi-hereditary.
\end{proposition}

\begin{proof}
Follows from Proposition \ref{proposition.5} and the proofs of Corollaries \ref{corollary.1} and \ref{corollary.2}. However, for the convenience of the reader, let us write down the full argument.

Suppose that $\Lambda$ is a right semi-hereditary ring. Let $P$ be a finitely generated projective right $\Lambda$-module, and consider the subcategory $\add P$ of $\proj \Lambda$. Observe that, by Proposition \ref{proposition.5}, $\proj \Lambda$ has $0$-kernels, and because $\add P$ is closed under direct summands in $\proj \Lambda$, it is clear that a $0$-kernel of a morphism in $\add P$ taken in $\proj \Lambda$ is also a $0$-kernel of it in $\add P$, so that $\add P$ has $0$-kernels. Furthermore, recall that, if $\End (P)$ is the endomorphism ring of $P$, then the categories $\add P$ and $\proj \End (P)$ are equivalent, see \cite[Section 8]{2409.10438}. Therefore, $\proj \End (P)$ has $0$-kernels, and we conclude from Proposition \ref{proposition.5} that $\End (P)$ is right semi-hereditary.
\end{proof}

Note that, by taking $\Lambda^{\op}$ in place of $\Lambda$ in Proposition \ref{proposition.27}, we obtain that if $\Lambda$ is left semi-hereditary, then the endomorphism ring of every finitely generated projective left $\Lambda$-module is right semi-hereditary. In addition to this result, we also have:

\begin{proposition}\label{proposition.28}
If $\Lambda$ is a left semi-hereditary ring, then the endomorphism ring of every finitely generated projective right $\Lambda$-module is left semi-hereditary.
\end{proposition}

\begin{proof}
This is similar to the proof of Proposition \ref{proposition.27}, the only difference being that, instead of Proposition \ref{proposition.5}, we must use its dual result.
\end{proof}

As before, observe that, by replacing $\Lambda$ by $\Lambda^{\op}$ in Proposition \ref{proposition.28}, we conclude that if $\Lambda$ is right semi-hereditary, then the endomorphism ring of  every finitely generated projective left $\Lambda$-module is left semi-hereditary.

\begin{corollary}\label{corollary.8}
Let $\Lambda$ be a right coherent ring. Then $\Lambda$ semi-hereditary if and only if every $M \in \mod \Lambda$ has a decomposition $M \simeq N \oplus P$ in $\mod \Lambda$, where $N \in \mod \Lambda$ is such that $\Hom_{\Lambda}(N,\Lambda) = 0$ and $P \in \proj \Lambda$.
\end{corollary}

\begin{proof}
Follows from Theorem \ref{theorem.7}, Proposition \ref{proposition.13} and the equivalence $\mod (\proj \Lambda) \approx \mod \Lambda$.\footnote{Note that, for $F \in \mod (\proj \Lambda)$, we have $F \in \eff (\proj \Lambda)$ if and only if $\Hom_{\Lambda}(F(\Lambda),\Lambda) = 0$.}
\end{proof}

Observe that, by taking $\Lambda^{\op}$ in place of $\Lambda$ in Corollary \ref{corollary.8}, we obtain the corresponding result for left $\Lambda$-modules. We leave its precise statement to the reader.

\appendix

\section{More on \texorpdfstring{$0$}{0}-kernels and \texorpdfstring{$0$}{0}-cokernels}\label{section.6}

In this short appendix, we present a couple of basic results concerning $0$-kernels and $0$-cokernels in arbitrary categories, which were mentioned in Subsection \ref{subsection.2.1}. Further basic results will be discussed in the upcoming paper \cite{0abelian}.

\begin{proposition}\label{proposition.2}
Let $\B$ be an arbitrary category. The following are equivalent:
\begin{enumerate}
    \item[(a)] $\B$ is idempotent complete.
    \item[(b)] Every idempotent in $\B$ has a $0$-kernel.
    \item[(c)] Every idempotent in $\B$ has a $0$-cokernel.
\end{enumerate}
\end{proposition}

\begin{proof}
Let $f$ be a morphism in $\B$ which is idempotent, that is, $f^{2} = f$.

If $\B$ is idempotent complete, then $f$ can be written as $f = gh$, where $g$ and $h$ are morphisms in $\B$ such that $hg = 1$. In this case, $g$ is a $0$-kernel of $f$ and $h$ is a $0$-cokernel of $f$. Therefore, item (a) implies both (b) and (c).

If $f$ has a $0$-kernel, then there is a decomposition $f = gh$ in $\B$, where $h$ is a split epimorphism and $g$ is a monomorphism. Thus, from $ghgh = gh$, we obtain that $hg = 1$. Consequently, item (b) implies (a). Similarly, (c) implies (a).
\end{proof}

Perhaps, Proposition \ref{proposition.2} sounds familiar to the reader, as it is well known that a preadditive category is idempotent complete if and only if every idempotent in it has a kernel if and only if every idempotent in it has a cokernel. However, note that, in Proposition \ref{proposition.2}, it is not necessary to assume that the category is preadditive.

Next, we present a description of $0$-kernels in terms of a universal property, which is in some sense analogous to the universal property of kernels. By duality, a similar description of $0$-cokernels also holds, as we point out below.

\begin{proposition}\label{proposition.10}
Let $\B$ be an arbitrary category, and let $f \in \B(X,Y)$ and $g \in \B(Z,Y)$ be morphisms in $\B$. Then $g$ is a $0$-kernel of $f$ if and only if it satisfies the following conditions:
\begin{enumerate}
    \item[(a)] There is some morphism $g' \in \B(Z,X)$ such that $g = fg'$.
    \item[(b)] If $w \in \B(W,Y)$ is a morphism for which there is some $w' \in \B(W,X)$ such that $w = fw'$, then there is a unique morphism $u \in \B(W,Z)$ satisfying $w = gu$.
\end{enumerate}
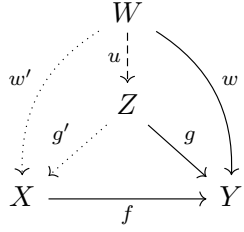
\[ \begin{tikzcd}
                   & W \arrow[rdd, "w", bend left] \arrow[ldd, "w'"', dotted, bend right] \arrow[d, "u"', dashed] &   \\
                   & Z \arrow[rd, "g"] \arrow[ld, "g'"', dotted]                                                  &   \\
X \arrow[rr, "f"'] &                                                                                              & Y
\end{tikzcd} \]
\end{proposition}

\begin{proof}
Suppose that $g$ is a $0$-kernel of $f$. Then $g$ is a monomorphism and there is a split epimorphism $h \in \B(X,Z)$ for which $f = gh$. Let $h' \in \B(Z,X)$ be such that $hh' = 1$. Then $g = fh'$, and condition (a) is satisfied. Next, let $w \in \B(W,Y)$ and $w' \in \B(W,X)$ be such that $w = fw'$. Then $w = ghw'$, and if $u \in \B(W,Z)$ is possibly another morphism satisfying $w = gu$, then it follows that $u = hw'$ as $g$ is a monomorphism. Hence condition (b) also holds.

Conversely, assume that $g$ satisfies conditions (a) and (b), so that there is a morphism $g' \in \B(Z,X)$ such that $g = fg'$. First, we will prove that $g$ is a monomorphism. Let $v,w \in \B(V,Z)$ be morphisms in $\B$ such that $gv = gw$. Then $gv = fg'v$, which implies that there is a unique morphism $u \in \B(V,Z)$ for which $gv = gu$. Thus, as $gv = gw$, we get that $v = w$. Finally, because $f = f 1$, it follows from condition (b) that there is a unique morphism $h \in \B(X,Z)$ satisfying $f = gh$. In this case, as $g = fg'$, we obtain that $g = ghg'$, which implies that $hg' = 1$ since $g$ is a monomorphism. Therefore, $h$ is a split epimorphism, and we conclude that $g$ is a $0$-kernel of $f$.
\end{proof}

Briefly, Proposition \ref{proposition.10} says that a $0$-kernel of a morphism $f$ is a morphism with the same codomain which factors through $f$ and is universal among all such morphisms. Since $0$-kernels and $0$-cokernels are concepts that are dual to each other, we can take opposite categories in Proposition \ref{proposition.10} to conclude a similar result for $0$-cokernels. In summary, it says that a $0$-cokernel of a morphism $f$ is a morphism with the same domain that factors through $f$ and is universal among all such morphisms.

\section{Relation to \texorpdfstring{$0$}{0}-Auslander categories}\label{section.7}

In this appendix, we point out how the concept of a $0$-abelian category relates to the notion of a \textit{$0$-Auslander extriangulated category}, which was recently introduced by Gorsky, Nakaoka and Palu in \cite[Definition 3.7]{Gorsky-Nakaoka-Palu}, see also \cite[Definition 3.1]{FGPPP}. We refer the reader to \cite{MR3931945} for the definition of an \textit{extriangulated category}.

Let us define a \textit{$0$-Auslander abelian category} to be an abelian category which is a $0$-Auslander extriangulated category when considered with its standard (abelian) extriangulated structure. In explicit terms, a category $\A$ is a $0$-Auslander abelian category if and only if $\A$ is an abelian category with enough projectives such that $\gldim \A \leqslant 1 \leqslant \domdim \A$. In the next two propositions, we show how such categories are related to $0$-abelian categories. We begin with:

\begin{proposition}\label{proposition.34}
There is a bijective correspondence between the equivalence classes of $0$-abelian categories with enough injectives and the equivalence classes of $0$-Auslander abelian categories for which their subcategories of projective objects are covariantly finite.\footnote{By the ``equivalence class of a category'', we mean the class of categories that are equivalent to it.} The correspondence is given as follows:
\begin{enumerate}
    \item[(a)] If $\C$ is a $0$-abelian category with enough injectives, then send it to $\mod \C$.
    \item[(b)] If $\A$ is a $0$-Auslander abelian category such that its subcategory $\Projectives (\A)$ of projective objects is covariantly finite in $\A$, then send it to $\Projectives (\A)$.
\end{enumerate}
\end{proposition}

\begin{proof}
It is well known that the assignment of $\C$ to $\mod \C$ gives a bijective correspondence between the equivalence classes of additive, idempotent complete and right coherent categories $\C$ and the equivalence classes of abelian categories with enough projectives, see, for example, \cite[Proposition 2.1.15]{MR4327095}. Moreover, the inverse of this correspondence assigns the equivalence class of an abelian category $\A$ with enough projectives to the equivalence class of $\Projectives (\A)$.

Now, observe that, by \cite[Lemma 5.2]{MR1856980}, if $\A$ is an abelian category with enough projectives, then $\Projectives (\A)$ is covariantly finite in $\A$ if and only if $\Projectives (\A)$ has weak cokernels, which is the case if and only if $\Projectives (\A)$ is left coherent, by \cite[Proposition 3.1]{2409.10438}. In particular, if $\C$ is an additive, idempotent complete  and right coherent category, then $\Projectives (\mod \C) = \proj \C$ is covariantly finite in $\mod \C$ if and only if $\C$ is left coherent, as the Yoneda embedding induces an equivalence $\C \approx \proj \C$. It follows from this observation and Theorem \ref{theorem.2} that, by restricting the above correspondence to the equivalence classes of $0$-abelian categories $\C$ with enough injectives, we obtain the desired correspondence with the equivalence classes of abelian categories $\A$ which are $0$-Auslander and such that $\Projectives (\A)$ is covariantly finite in $\A$.
\end{proof}

The second proposition, which is given below, is obtained by dropping the assumption of Proposition \ref{proposition.34} that the subcategories of projective objects of the $0$-Auslander abelian categories are covariantly finite.

\begin{proposition}
There is a bijective correspondence between the equivalence classes of additive categories that have $0$-kernels and enough injectives, and the equivalence classes of $0$-Auslander abelian categories. The correspondence is given as follows:
\begin{enumerate}
    \item[(a)] If $\B$ is an additive category that has $0$-kernels and enough injectives, then send it to $\mod \B$.
    \item[(b)] If $\A$ is a $0$-Auslander abelian category, then send it to $\Projectives (\A)$, its subcategory of projective objects.
\end{enumerate}
\end{proposition}

\begin{proof}
This is similar to the proof of Proposition \ref{proposition.1}, but follows from the fact that an additive category $\B$ has $0$-kernels and enough injectives if and only if $\B$ is idempotent complete, right coherent and $\gldim (\mod \B) \leqslant 1 \leqslant \domdim (\mod \B)$. To see that this fact holds, note that, by Propositions \ref{proposition.2} and \ref{proposition.1}, an additive category $\B$ has $0$-kernels if and only if $\B$ is idempotent complete, right coherent and $\gldim (\mod \B) \leqslant 1$. Furthermore, it follows from the proof of Proposition \ref{proposition.3} that if $\B$ is additive, idempotent complete, right coherent and $\gldim (\mod \B) \leqslant 1$, then $\B$ is right comprehensive. Therefore, if $\B$ satisfies these previous conditions, then $\B$ has enough injectives if and only if $\domdim (\mod \B) \geqslant 1$, by \cite[Corollary 7.5]{2409.10438}.
\end{proof}

Finally, we remark that, in light of results of Chen from \cite{Chen.2306} and \cite{Chen.2308}, it becomes natural to ask whether there exists a notion of a ``$0$-cluster tilting subcategory'' of an abelian category. One would expect that such subcategories coincide with $0$-abelian categories, in the same way that $n$-cluster tilting subcategories of abelian categories coincide with $n$-abelian categories, when $n \geqslant 1$, as proved by Kvamme in \cite[Corollary 1.3]{MR4301013} and by Ebrahimi and Nasr-Isfahani in \cite[Theorem A]{MR4514466}. The author believes that the search for the notion of a ``$0$-cluster tilting subcategory'' would lead to exciting new research.

\section*{Acknowledgments}

The author would like to thank his Ph.D. advisor Alex Martsinkovsky for being receptive to the ideas in this paper, and for his many comments and suggestions. The author also thanks Hipolito Treffinger for stimulating discussions on torsion pairs, which led to the content of Section \ref{section.3}.


\begin{thebibliography}{10}

\bibitem{MR1244922}
J.~Asensio~Mayor and J.~Mart\'inez~Hern\'andez.
\newblock On flat and projective envelopes.
\newblock {\em J. Algebra}, 160(2):434--440, 1993.

\bibitem{MR2197389}
Ibrahim Assem, Daniel Simson, and Andrzej Skowro\'nski.
\newblock {\em Elements of the representation theory of associative algebras. {V}ol. 1}, volume~65 of {\em London Mathematical Society Student Texts}.
\newblock Cambridge University Press, Cambridge, 2006.
\newblock Techniques of representation theory.

\bibitem{MR4860676}
Jenny August, Johanne Haugland, Karin~M. Jacobsen, Sondre Kvamme, Yann Palu, and Hipolito Treffinger.
\newblock A characterisation of higher torsion classes.
\newblock {\em Forum Math. Sigma}, 13:Paper No. e33, 36, 2025.

\bibitem{MR0212070}
Maurice Auslander.
\newblock Coherent functors.
\newblock In {\em Proc. {C}onf. {C}ategorical {A}lgebra ({L}a {J}olla, {C}alif., 1965)}, pages 189--231. Springer-Verlag New York, Inc., New York, 1966.

\bibitem{MR237606}
Maurice Auslander.
\newblock Comments on the functor {${\rm Ext}$}.
\newblock {\em Topology}, 8:151--166, 1969.

\bibitem{MR404240}
Maurice Auslander and Idun Reiten.
\newblock Stable equivalence of dualizing {$R$}-varieties. {II}. {H}ereditary dualizing {$R$}-varieties.
\newblock {\em Advances in Math.}, 17(2):93--121, 1975.

\bibitem{MR816889}
Maurice Auslander and Idun Reiten.
\newblock Grothendieck groups of algebras and orders.
\newblock {\em J. Pure Appl. Algebra}, 39(1-2):1--51, 1986.

\bibitem{MR591246}
Maurice Auslander and Sverre~O. Smal{\o}.
\newblock Preprojective modules over {A}rtin algebras.
\newblock {\em J. Algebra}, 66(1):61--122, 1980.

\bibitem{MR2027559}
Apostolos Beligiannis.
\newblock On the {F}reyd categories of an additive category.
\newblock {\em Homology Homotopy Appl.}, 2:147--185, 2000.

\bibitem{MR1856980}
Apostolos Beligiannis.
\newblock Homotopy theory of modules and {G}orenstein rings.
\newblock {\em Math. Scand.}, 89(1):5--45, 2001.

\bibitem{MR1291599}
Francis Borceux.
\newblock {\em Handbook of categorical algebra. 1}, volume~50 of {\em Encyclopedia of Mathematics and its Applications}.
\newblock Cambridge University Press, Cambridge, 1994.
\newblock Basic category theory.

\bibitem{Chen.2306}
Xiaofa Chen.
\newblock 0-{A}uslander correspondence, 2023.
\newblock \href{https://arxiv.org/abs/2306.15958}{arXiv:2306.15958}.

\bibitem{Chen.2308}
Xiaofa Chen.
\newblock Auslander--{I}yama correspondence for exact dg categories, 2023.
\newblock \href{https://arxiv.org/abs/2308.08519}{arXiv:2308.08519}.

\bibitem{MR269686}
Robert~R. Colby and Edgar~A. Rutter, Jr.
\newblock Generalizations of {${\rm QF}-3$} algebras.
\newblock {\em Trans. Amer. Math. Soc.}, 153:371--386, 1971.

\bibitem{MR191935}
Spencer~E. Dickson.
\newblock A torsion theory for {A}belian categories.
\newblock {\em Trans. Amer. Math. Soc.}, 121:223--235, 1966.

\bibitem{MR1248754}
Nan~Qing Ding and Jian~Long Chen.
\newblock Relative coherence and preenvelopes.
\newblock {\em Manuscripta Math.}, 81(3-4):243--262, 1993.

\bibitem{MR202792}
Clifford~H. Dowker.
\newblock Composite morphisms in abelian categories.
\newblock {\em Quart. J. Math. Oxford Ser. (2)}, 17:98--105, 1966.

\bibitem{EbrahimiNasr-Isfahani}
Ramin Ebrahimi and Alireza Nasr-Isfahani.
\newblock Higher auslander correspondence for exact categories, 2021.
\newblock \href{https://arxiv.org/abs/2108.13645}{arXiv:2108.13645}.

\bibitem{MR4514466}
Ramin Ebrahimi and Alireza Nasr-Isfahani.
\newblock Higher {A}uslander's formula.
\newblock {\em Int. Math. Res. Not. IMRN}, (22):18186--18203, 2022.

\bibitem{MR1753146}
Edgar~E. Enochs and Overtoun M.~G. Jenda.
\newblock {\em Relative homological algebra}, volume~30 of {\em De Gruyter Expositions in Mathematics}.
\newblock Walter de Gruyter \& Co., Berlin, 2000.

\bibitem{FGPPP}
Xin Fang, Mikhail Gorsky, Yann Palu, Pierre-Guy Plamondon, and Matthew Pressland.
\newblock Extriangulated ideal quotients, with applications to cluster theory and gentle algebras, 2024.
\newblock \href{https://arxiv.org/abs/2308.05524}{arXiv:2308.05524}.

\bibitem{MR206069}
Peter Freyd.
\newblock Splitting homotopy idempotents.
\newblock In {\em Proc. {C}onf. {C}ategorical {A}lgebra ({L}a {J}olla, {C}alif., 1965)}, pages 173--176. Springer-Verlag New York, Inc., New York, 1966.

\bibitem{MR322004}
Peter~J. Freyd and Gregory~M. Kelly.
\newblock Categories of continuous functors. {I}.
\newblock {\em J. Pure Appl. Algebra}, 2:169--191, 1972.

\bibitem{MR859386}
Jos\'e~L. Garc\'ia~Hern\'andez and Jos\'e~L. G\'omez~Pardo.
\newblock Hereditary and semihereditary endomorphism rings.
\newblock In {\em Ring theory ({A}ntwerp, 1985)}, volume 1197 of {\em Lecture Notes in Math.}, pages 83--89. Springer, Berlin, 1986.

\bibitem{Gorsky-Nakaoka-Palu}
Mikhail Gorsky, Hiroyuki Nakaoka, and Yann Palu.
\newblock Hereditary extriangulated categories: Silting objects, mutation, negative extensions, 2023.
\newblock \href{https://arxiv.org/abs/2303.07134}{arXiv:2303.07134}.

\bibitem{0abelian}
Vitor Gulisz.
\newblock {$0$}-{A}belian categories.
\newblock In preparation.

\bibitem{2409.10438}
Vitor Gulisz.
\newblock A functorial approach to $n$-abelian categories, 2024.
\newblock \href{https://arxiv.org/abs/2409.10438}{arXiv:2409.10438}.

\bibitem{MR4392222}
Ruben Henrard, Sondre Kvamme, and Adam-Christiaan van Roosmalen.
\newblock Auslander's formula and correspondence for exact categories.
\newblock {\em Adv. Math.}, 401:Paper No. 108296, 65, 2022.

\bibitem{MR4575371}
Ruben Henrard, Sondre Kvamme, Adam-Christiaan van Roosmalen, and Sven-Ake Wegner.
\newblock The left heart and exact hull of an additive regular category.
\newblock {\em Rev. Mat. Iberoam.}, 39(2):439--494, 2023.

\bibitem{MR124377}
Peter~J. Hilton and David Rees.
\newblock Natural maps of extension functors and a theorem of {R}. {G}. {S}wan.
\newblock {\em Proc. Cambridge Philos. Soc.}, 57:489--502, 1961.

\bibitem{MR3519980}
Gustavo Jasso.
\newblock {$n$}-abelian and {$n$}-exact categories.
\newblock {\em Math. Z.}, 283(3-4):703--759, 2016.

\bibitem{MR4327095}
Henning Krause.
\newblock {\em Homological theory of representations}, volume 195 of {\em Cambridge Studies in Advanced Mathematics}.
\newblock Cambridge University Press, Cambridge, 2022.

\bibitem{MR4301013}
Sondre Kvamme.
\newblock Axiomatizing subcategories of {A}belian categories.
\newblock {\em J. Pure Appl. Algebra}, 226(4):Paper No. 106862, 27, 2022.

\bibitem{MR1653294}
Tsit~Yuen Lam.
\newblock {\em Lectures on modules and rings}, volume 189 of {\em Graduate Texts in Mathematics}.
\newblock Springer-Verlag, New York, 1999.

\bibitem{MR170924}
Johann~B. Leicht.
\newblock {\"U}ber die elementaren {L}emmata der homologischen {A}lgebra in quasi-exacten {K}ategorien.
\newblock {\em Monatsh. Math.}, 68:240--254, 1964.

\bibitem{MR280520}
Helmut Lenzing.
\newblock Halberbliche {E}ndomorphismenringe.
\newblock {\em Math. Z.}, 118:219--240, 1970.

\bibitem{MR1712872}
Saunders Mac~Lane.
\newblock {\em Categories for the working mathematician}, volume~5 of {\em Graduate Texts in Mathematics}.
\newblock Springer-Verlag, New York, second edition, 1998.

\bibitem{MR224649}
B.~H. Maddox.
\newblock Absolutely pure modules.
\newblock {\em Proc. Amer. Math. Soc.}, 18:155--158, 1967.

\bibitem{MR3516080}
Alex Martsinkovsky.
\newblock On direct summands of homological functors on length categories.
\newblock {\em Appl. Categ. Structures}, 24(4):421--431, 2016.

\bibitem{MR3335998}
Alex Martsinkovsky and Dali Zangurashvili.
\newblock The stable category of a left hereditary ring.
\newblock {\em J. Pure Appl. Algebra}, 219(9):4061--4089, 2015.

\bibitem{MR306265}
D.~George McRae.
\newblock Homological dimensions of finitely presented modules.
\newblock {\em Math. Scand.}, 28:70--76, 1971.

\bibitem{MR294409}
Charles Megibben.
\newblock Absolutely pure modules.
\newblock {\em Proc. Amer. Math. Soc.}, 26:561--566, 1970.

\bibitem{MR233854}
Bruno~J. M\"uller.
\newblock Dominant dimension of semi-primary rings.
\newblock {\em J. Reine Angew. Math.}, 232:173--179, 1968.

\bibitem{MR3931945}
Hiroyuki Nakaoka and Yann Palu.
\newblock Extriangulated categories, {H}ovey twin cotorsion pairs and model structures.
\newblock {\em Cah. Topol. G\'eom. Diff\'er. Cat\'eg.}, 60(2):117--193, 2019.

\bibitem{MR2563182}
Pedro Nicol\'as and Manuel Saor\'in.
\newblock Balance in stable categories.
\newblock {\em Algebr. Represent. Theory}, 12(6):543--566, 2009.

\bibitem{MR2471947}
Wolfgang Rump.
\newblock A counterexample to {R}aikov's conjecture.
\newblock {\em Bull. Lond. Math. Soc.}, 40(6):985--994, 2008.

\bibitem{MR258888}
Bo~Stenstr\"om.
\newblock Coherent rings and {$FP$}-injective modules.
\newblock {\em J. London Math. Soc. (2)}, 2:323--329, 1970.

\bibitem{MR349740}
Hiroyuki Tachikawa.
\newblock {\em Quasi-{F}robenius rings and generalizations. {${\rm QF}-3$} and {${\rm QF}-1$} rings}, volume Vol. 351 of {\em Lecture Notes in Mathematics}.
\newblock Springer-Verlag, Berlin-New York, 1973.
\newblock Notes by Claus Michael Ringel.

\bibitem{MR2373322}
Adam-Christiaan van Roosmalen.
\newblock Classification of abelian hereditary directed categories satisfying {S}erre duality.
\newblock {\em Trans. Amer. Math. Soc.}, 360(5):2467--2503, 2008.

\bibitem{MR268222}
Charles Vinsonhaler.
\newblock Supplement to the paper: ``{O}rders in {${\rm QF}-3$} rings''.
\newblock {\em J. Algebra}, 17:149--151, 1971.

\end{thebibliography}
\end{document}